\documentclass[11pt]{amsart}

\usepackage[utf8]{inputenc}
\usepackage[english]{babel}
\usepackage[a4paper,twoside,top=1.2in, bottom=1.2in, left=0.8in, right=0.8in]{geometry}

\usepackage{graphicx}
\usepackage{caption} 
\usepackage{amsmath}
\usepackage{amsthm}
\usepackage{amssymb}
\usepackage{esint} 
\usepackage{mathrsfs} 
\usepackage{xcolor}
\usepackage{array}
\usepackage{hhline}
\usepackage{enumitem} 
\usepackage{comment} 
\usepackage[toc,page]{appendix} 
\usepackage{xparse} 
\usepackage{mathtools}
\usepackage{fancyhdr} 
\usepackage{ifthen} 
\usepackage{forloop} 
\usepackage{xstring}
\usepackage{emptypage} 
\usepackage{setspace}
\usepackage[initials,alphabetic]{amsrefs} 

\usepackage{tikz}
\usetikzlibrary{arrows,shapes,patterns,calc,fadings,decorations.pathreplacing,decorations.markings,decorations.pathmorphing,backgrounds}

\usepackage{hyperref} 
\hypersetup{
	colorlinks = true,
	linkcolor = {blue},
	urlcolor = {red},
	citecolor = {blue}
}

\usepackage[nameinlink,capitalise]{cleveref} 

\newcounter{results}[section] 

\theoremstyle{plain}
\newtheorem{theorem}[results]{Theorem}
\newtheorem{lemma}[results]{Lemma}
\newtheorem{proposition}[results]{Proposition}
\newtheorem{corollary}[results]{Corollary}

\newtheorem*{theorem*}{Theorem}
\newtheorem*{lemma*}{Lemma}
\newtheorem*{proposition*}{Proposition}
\newtheorem*{corollary*}{Corollary}
\newtheorem*{exercise*}{Exercise}
\newtheorem*{fact*}{Fact}

\theoremstyle{remark}
\newtheorem{remark}[results]{Remark}

\newtheorem*{remark*}{Remark}
\newtheorem*{question*}{Question}

\theoremstyle{definition}

\newtheorem*{definition*}{Definition}
\newtheorem*{example*}{Example}

\numberwithin{equation}{section}

\crefname{figure}{Figure}{Figures}

\newcommand{\Z}{\ensuremath{\mathbb Z}}
\newcommand{\R}{\ensuremath{\mathbb R}}
\newcommand{\C}{\ensuremath{\mathbb C}}

\newcommand{\co}[2]{\ensuremath{\ClOp{#1, #2}}}

\newcommand{\bigO}{\ensuremath{\mathcal{O}}}

\DeclarePairedDelimiter\abs{\lvert}{\rvert} 
\DeclarePairedDelimiter\bigabs{\big\lvert}{\big\rvert} 
\DeclarePairedDelimiter\norm{\lVert}{\rVert} 
\newcommand{\scal}[2]{\ensuremath{\langle #1 , #2 \rangle}} 

\newcommand{\eqdef}{\ensuremath{\coloneqq}} 
\newcommand\restr[2]{{
  \left.\kern-\nulldelimiterspace 
  #1 
  \vphantom{\big|} 
  \right|_{#2} 
  }}

\newcommand{\DerParz}[2]{\ensuremath{\frac{\partial #1}{\partial #2}}} 
\DeclareMathOperator{\II}{I\!I} 
\DeclareMathOperator{\dist}{dist} 

\DeclareMathOperator{\Iso}{Iso} 
\DeclareMathOperator{\U}{U} 
\DeclareMathOperator{\D}{D} 
\newcommand{\EBall}[1]{\ensuremath{ B^{\R^3}_{#1}(0)}} 
\newcommand{\EBallc}[2]{\ensuremath{ B^{\R^3}_{#1}(#2)}} 
\newcommand{\G}{\ensuremath{\mathcal{G}}} 
\newcommand{\Sc}{\ensuremath{\mathcal{S}}} 
\newcommand{\ScIn}{\ensuremath{\mathcal{S}_{T_1,T_2}^d}}
\DeclareMathOperator{\graph}{graph} 
\renewcommand{\L}{\mathcal{L}} 
\newcommand{\Ctwoa}{\ensuremath{\mathcal{C}^{2,\alpha}}} 
\newcommand{\Ctwob}{\ensuremath{\mathcal{C}^{2,\beta}}} 
\newcommand{\Cka}{\ensuremath{\mathcal{C}^{k,\alpha}}} 
\newcommand{\Czeroa}{\ensuremath{\mathcal{C}^{0,\alpha}}} 
\newcommand{\Czerob}{\ensuremath{\mathcal{C}^{0,\beta}}} 
\newcommand{\Czero}{\ensuremath{\mathcal{C}^{0}}} 
\newcommand{\loc}{\ensuremath{loc}} 
\newcommand{\TN}{\mathit{TN}} 
\DeclareMathOperator{\Ind}{{Ind}} 
\DeclareMathOperator{\Null}{{Null}} 
\DeclareMathOperator{\im}{\ensuremath{\mathrm{im}}}
\newcommand{\T}{\ensuremath{\mathbb T}}

\colorlet{myGray}{gray}
\colorlet{myBlue}{blue}
\colorlet{myBlack}{black}
\colorlet{myBackground}{gray!10}

\def\co{\colon\thinspace}

\title[]{Unstable minimal spheres with degree-1 Gauss lift in hyperk\"ahler 4-manifolds}
\author{Lorenzo Foscolo and Federico Trinca}
\newcommand\printaddress{{
\setlength{\parindent}{17pt}
\footnotesize

\bigskip
\par 
{\scshape \noindent Lorenzo Foscolo}
\newline Sapienza Università di Roma, Piazzale A. Moro 5, 00185 Roma, Italia
\newline
\textit{E-mail address:} \texttt{lorenzo.foscolo@uniroma1.it}
\newline
\par
{\scshape \noindent Federico Trinca}
\newline Department of Mathematics, UBC, 1984 Mathematics Road, Vancouver, Canada
\newline	
\textit{E-mail address:} \texttt{ftrinca@math.ubc.ca}
\par
}} 

\begin{document}
\begin{abstract}
We exhibit new minimal 2-spheres in hyperk\"ahler 4-manifolds arising from the Gibbons--Hawking ansatz
and in the K3 manifold endowed with a hyperk\"ahler metric. These minimal surfaces are obtained via a gluing construction using the Scherk surface in flat space and the holomorphic cigar in the Taub-NUT space as building blocks. As for the stable minimal 2-sphere in the Atiyah--Hitchin manifold, the minimal surfaces we construct are not holomorphic with respect to any complex structure compatible with the metric, have degree-1 positive Gauss lift so they can be parametrised by a harmonic map that satisfies a first-order Fueter-type PDE, and yet are unstable. This shows that there is no characterisation of stable minimal surfaces in hyperk\"ahler 4-manifolds in terms of topological data.
\end{abstract}
\maketitle

\thispagestyle{empty}

\section{Introduction}

It is well known that holomorphic submanifolds of a K\"ahler manifold minimise volume in their homology class. A classical problem in minimal surface theory is to understand to what extent area minimising surfaces (and more generally stable minimal surfaces) in K\"ahler manifolds must be (anti)holomorphic. For example, in 1993 Yau asked whether it is possible to classify all stable minimal 2-spheres in a simply connected K\"ahler Ricci-flat manifold \cite[Question 64]{Yau}. In \cite{Micallef} Micallef showed that every stable minimal surface in a flat 4–torus must be holomorphic for some complex structure compatible with the metric. (Note however that this is no longer the case for higher dimensional tori \cite{Arezzo:Micallef}.) In view of Micallef's result it was thought for some time that a similar result could hold for the K3 manifold (the smooth 4-manifold underlying any complex K3 surface) endowed with a hyperk\"ahler metric. Partial results in this direction were established by Micallef--Wolfson \cite[Theorem 5.3]{Micallef:Wolfson} and motivation for the conjecture came from the fact that, given an arbitrary hyperk\"ahler metric on the K3 manifold, every homology class can be represented by the sum of surfaces each of which is holomorphic with respect to some complex structure compatible with the metric. However, Micallef--Wolfson \cite{Micallef:Wolfson:K3} have eventually shown that no analogue of the result for 4–tori holds for the K3 manifold. Indeed they found a class $\alpha\in H^2(K3;\Z)$ and a hyperk\"ahler metric $g$ on the K3 manifold such that the volume minimiser in $\alpha$ decomposes into a sum of branched minimal surfaces $\Sigma_1\cup\dots\cup\Sigma_k$ not all of which can be holomorphic with respect to some complex structure compatible with $g$.

Further more explicit counterexamples were produced by the first named author as an application of his construction of hyperk\"ahler metrics on the K3 manifold degenerating to a 3-dimensional limit \cite{Foscolo2019}. Indeed, the simplest stable (in fact, area minimising) minimal 2-sphere in a hyperk\"ahler 4-manifold that is not holomorphic with respect to any complex structure compatible with the metric is given by the minimal 2-sphere at the core of (the double cover of) the Atiyah--Hitchin manifold. The latter is a complete non-compact hyperk\"ahler 4-manifold with an isometric cohomogeneity one action of $SU(2)$ constructed by Atiyah--Hitchin \cite{Atiyah:Hitchin}. The Atiyah--Hitchin manifold retracts to a 2-sphere of self-intersection $-4$, the unique singular orbit for the $SU(2)$--action. The fact that this $2$-sphere is a strictly stable minimal surface is proved in \cite[Proposition 5.5]{Micallef:Wolfson} and the stronger area minimising (and strong stability) property was observed more recently by Tsai--Wang \cite[Proposition 3.4]{Tsai:Wang}. Because of the self-intersection number, the adjunction formula shows that that this minimal 2-sphere cannot be holomorphic with respect to any complex structure compatible with the metric. Since the construction of hyperk\"ahler metrics on the K3 manifold of \cite{Foscolo2019} uses the Atiyah--Hitchin metric as one of the building blocks in a gluing construction, the existence of non-holomorphic strictly stable minimal 2-spheres in the K3 manifold follows by perturbation.

By a classical result of Eells--Salamon \cite{Eells:Salamon}, minimal surfaces in $4$-manifolds can always be interpreted as (anti)holomorphic objects by passing to the twistor space. Since the twistor space of a hyperk\"ahler $4$-manifold $(M,g)$ is $Z=M\times S^2$, the Eells--Salamon twistor correspondence has an explicit interpretation. Given any minimal immersion $u\co \Sigma\rightarrow M$ we have a map $a\co \Sigma\rightarrow S^2$, which we call the positive Gauss lift of $u$, defined by the property that for any $p\in \Sigma$ the image $a(p)$ is the opposite of the unique self-dual $2$-form up to scale on $T_{u(p)}M$ that calibrates the tangent plane to $u (\Sigma)$. Endowing $\Sigma$ with the conformal structure $j_0$ induced by $u^\ast g$ and $S^2$ with its standard complex structure, the fact that $u(\Sigma)$ is a minimal surface is equivalent to the fact that $a$ is holomorphic and $u$ satisfies
\begin{equation}\label{eq:Fueter}
du\circ j_0 + (a_1 J_1 \circ du+a_2 J_2 \circ du+a_3 J_3 \circ du) =0,
\end{equation}
where we think of $S^2$ as the unit $2$-sphere in $\R^3$ and $(J_1,J_2,J_3)$ is the triple of complex structures compatible with the hyperk\"ahler metric $g$. Moreover, the self-intersection number and genus $\gamma$ of $\Sigma$ are related to the degree of $a$ by
\begin{equation}\label{eq:Webster}
2\deg{a} + 2-2\gamma + [u(\Sigma)]\cdot [u(\Sigma)] = 0,
\end{equation}
often referred to as Webster's formula \cite{Webster1984}. In particular, if $u\co \Sigma=S^2\rightarrow M$ is a minimal $2$-sphere with self-intersection $-4$, then $a$ has degree $1$, we can reparametrise so that $a$ is the identity, and the first-order PDE \eqref{eq:Fueter} becomes a known first-order equation in harmonic map theory whose solutions induce radially invariant energy-minimising harmonic maps $f\co \R^3\setminus\{ 0\}\rightarrow M$. Such minimal $2$-spheres have been studied by Chen--Li \cite{Chen:Li} and have appeared more recently in compactness questions for Fueter and triholomorphic maps (and more generally sections of smooth fibrations) from $3$ and $4$-manifolds \cite{Bellettini:Tian,Walpuski,Esfahani}.

Now, the positive Gauss lift of a minimal surface in a hyperk\"ahler 4-manifold plays a similar role as the Gauss map of a minimal surface in $\R^3$. In the classical setting, known generalisations of the Bernstein Theorem to the classification of complete (stable) minimal surfaces in $\R^3$ in terms of ``topological'' properties of their Gauss map have been established for example by Ossermann \cite{Osserman}, do Carmo--Peng \cite{doCarmo:Peng}, Fischer-Colbrie--Schoen \cite{FischerColbrie:Schoen} and Pogorelov \cite{Pogorelov1981}. For surfaces in hyperk\"ahler 4-manifolds, besides the already cited \cite{Micallef}, interesting results relating stability, holomorphicity and the degree of the positive Gauss lift (equivalently, the self-intersection number) have been established by Arezzo, Micallef--Wolfson and Donaldson. Arezzo \cite{Arezzo} proved that a complete minimal surface is holomorphic if the image of the positive Gauss lift omits an open set of $S^2$. Micallef--Wolfson \cite[Corollary of Theorem 5.3]{Micallef:Wolfson} have shown that a minimal surface is holomorphic if and only if $\deg{a}=0$ (this in fact was proved earlier by Wolfson \cite[Theorem 2.2]{Wolfson}) and that a stable minimal surface is either holomorphic or satisfies $2\deg{a}\geq \gamma +2$. Equivalently, the self-intersection number of a stable minimal surface that is not holomorphic with respect to any complex structure compatible with the metric must be at most $\gamma -4$. Donaldson \cite[Proposition 26]{Donaldson} (see also \cite[Theorem 5.1]{Micallef:Wolfson:K3}) has further shown that a non-holomorphic strictly stable minimal surface in a hyperk\"ahler 4-manifold must have self-intersection number at most $-4$ independently of its genus. The minimal 2-sphere in the Atiyah--Hitchin manifold shows that this bound is sharp.

Given that the only explicitly known non-holomorphic stable minimal 2-spheres in hyperk\"ahler 4-manifolds all arise from the minimal 2-sphere in the Atiyah--Hitchin manifold, as a reformulation of Yau's question one could ask to what extent stable minimal 2-spheres in a hyperk\"ahler 4-manifold are characterised by the condition of having positive Gauss lift of degree $0$ or $1$. In this paper we exhibit minimal 2-spheres in hyperk\"ahler 4-manifolds, including certain hyperk\"ahler metrics on the K3 manifold, that have self-intersection $-4$, so they satisfy \eqref{eq:Fueter} with $a$ the identity, and yet are unstable. This proves that there is no topological characterisation of stable minimal 2-spheres in hyperk\"ahler 4-manifolds and ultimately that there is no possible classification along the lines of Yau's question. Our main existence results are \cref{thm: Main theorem,thm: spheres in K3}. Together with \cite{Foscolo2019}, the latter in particular implies
\begin{theorem*}
On the K3 manifold there exist hyperk\"ahler metrics that contain stable minimal $2$-spheres with positive Gauss lift of degree $1$ (equivalently, with self-intersection number $-4$) and hyperk\"ahler metrics that contain unstable minimal 2-spheres with the same topological properties.
\end{theorem*}


\begin{remark*}
We also produce minimal surfaces in hyperk\"ahler 4-manifolds obtained via the Gibbons--Hawking ansatz that have arbitrarily large Morse index, either by increasing their genus or the complexity of their normal bundle (in the latter case, the ambient geometry must also be of increasingly complicated topology), see \cref{thm: multiple periods,thm: saddle tower}.
\end{remark*}

\begin{remark*}
    Besides the mathematical motivations we have discussed, the holomorphicity of volume minimisers and stable minimal submanifolds in K\"ahler Ricci-flat manifolds have recently gained interest also in the theoretical physics literature \cite{WeakGravity} in the context of the Weak Gravity Conjecture.
\end{remark*}

The minimal surfaces we produce arise from a simple gluing construction. The simplest ambient geometry we consider is the complete non-compact hyperk\"ahler metrics with ALF asymptotics arising from the Gibbons--Hawking ansatz with 4 centres. Recall (more details will be given in \cref{sec:Gibbons:Hawking}) that the Gibbons--Hawking ansatz produces families of explicit hyperk\"ahler metrics with an isometric and triholomorphic circle action from a collection of $k$ distinct points in $\R^3$ and a constant $\ell\in (0,\infty]$ that determines the asymptotic geometry at infinity: if $\ell\in (0,\infty)$ the asymptotic geometry is ALF, \emph{i.e.} the one of a Riemannian submersion with base the flat metric on $\R^3$ and circle fibres of fixed length $2\pi\ell$; the limiting case $\ell = \infty$ corresponds instead to an ALE metric asymptotic to $\R^4/\Z_k$. The relative positions of the $k$ points in $\R^3$ can be interpreted instead as the periods of the  hyperk\"ahler triple. In our set-up we consider a family $X_d$ of ALF Gibbons--Hawking metrics corresponding to $\ell=1$ (which for $\ell\in (0,\infty)$ can always be achieved by scaling) and $k=4$ points with mutual distance $d\gg1$. In the limit $d\rightarrow \infty$, the geometry of $X_d$ is captured by five different pointed Gromov--Hausdorff limits: a copy of $\R^3\times S^1$ and four copies of the Taub--NUT metric on $\R^4$. Our gluing construction combines simple minimal surfaces in each of these five building blocks. First of all, we regard the singly-periodic \emph{Scherk surface}, one of the classical minimal surfaces in $\R^3$, as a minimal surface in (a totally geodesic copy of $\R^2\times S^1$ in) $\R^3\times S^1$ with four asymptotically cylindrical ends. Each of the ends is then capped off by a holomorphic (for a suitable complex structure) disc, the \emph{cigar}, in each of the four Taub--NUT spaces. For $d\gg 1$ sufficiently large, this produces an $S^1$ family of approximate minimal spheres in $X_d$, as described in \cref{sec: aprroximate minimal surface}.

Now, since the linearisation of the minimal surface equation is a self-adjoint operator, the presence of a $1$-parameter family of approximate solutions to the equation means that the gluing problem might be obstructed in general. We overcome this difficulty by considering the most symmetric situation where the four centres of the multi-Taub--NUT space $X_d$ lie at the vertices of a square with side-length $\sqrt{2} d$. We then have a large group of discrete symmetries in our construction and in \cref{sec: linear problem on building blocks} we show that the linearised operators of the building blocks of our approximate minimal surfaces have no equivariant kernel and cokernel. An application of (a quantitative version of) the Implicit Function Theorem then yields our first main result \cref{thm: Main theorem}, the existence of minimal spheres in $X_d$ for every $d$ sufficiently large. The fact that the Scherk surface has Gauss map of degree 1 (once we compactify it by adding four points corresponding to the unit normals to the cylindrical ends) implies the surfaces we produce have positive Gauss lift of degree $1$. Similarly, the fact that the Scherk surface has index 1 as a minimal surface in $\R^3$ is the reason why the minimal surfaces we produce have index at least $1$ (and nullity at least $1$, since the Killing field on $X_d$ generating the circle action yields a non-trivial Jacobi field on our minimal surfaces).

\begin{remark*}
Until now, all known compact minimal surfaces in hyperk\"ahler 4-manifolds arising from the Gibbons--Hawking ansatz were circle invariant holomorphic spheres. A convexity argument by the second named author shows that in the simplest case of $k=2$ centres these are all the compact minimal surfaces \cite[Theorem 1.1]{Trinca2022}. Our \cref{thm: Main theorem} shows this is not the case in general. 
\end{remark*}

\begin{remark*}
Besides the family of ALF metrics arising from the Gibbons--Hawking ansatz, also known as cyclic ALF metrics of type $A_3$, the smooth 4-manifold $X_d$ carries another family of ALF metrics of dihedral type $D_3$. Both families share Kronheimer's ALE metrics of type $A_3=D_3$ as a limit for $\ell\rightarrow\infty$, while they have $\R^3$ and, respectively, $\R^3/\Z_2$ as a collapsed limit for $\ell\rightarrow 0$. Some of the dihedral ALF metrics contain a strictly stable minimal $2$-sphere in the same homology class of our unstable minimal $2$-spheres. Indeed, following the suggestion by Sen \cite{Sen} that motivated \cite{Foscolo2019}, Schroers--Singer \cite{Schroers:Singer} and Zhu \cite{Zhu2024} have shown how to produce $D_3$ ALF metrics in a gluing construction using the (double cover of the) Atiyah--Hitchin metric as a building block; the existence of a stable minimal 2-sphere follows as in \cite{Foscolo2019}, see \cite{Zhu2024}.
\end{remark*}

It seems likely that a substantial generalisation of our construction could be given in the case where discrete symmetries are not imposed. Indeed, Kapouleas \cite{Kapouleas2011} has developed a general gluing construction of minimal hypersurfaces without discrete symmetries introducing suitable deformations of the Scherk surface to compensate for the $1$-dimensional kernel of the linearised operator. Such a generalisation would also require a better understanding of the analysis on the cigar minimal surface in the Taub--NUT space, which is asymptotically cylindrical but not with exponential decay. In this paper we preferred concentrating on the geometric aspects and consequences of our constructions, using discrete symmetries to eliminate most technical difficulties.

In \cref{sec:Applications} we give instead various immediate generalisations of our existence result using multiple periods of the Scherk surface (thus producing unstable minimal surfaces of any genus in $X_d$ with self-intersection number $-4$) and Karcher's saddle towers, certain minimal surfaces in $\R^2\times S^1$ with an arbitrary even number of cylindrical ends. More importantly, we use the construction of hyperk\"ahler metrics on the K3 manifold in \cite{Foscolo2019} to deduce our second main result \cref{thm: spheres in K3}, the existence of unstable minimal $2$-spheres with self-intersection $-4$ in the K3 manifold. This is achieved by finding a suitable highly symmetric configuration for the construction of \cite{Foscolo2019} where rescaled copies of $X_d$ and its minimal surfaces can be used as building blocks.

\subsubsection*{Acknowledgements} During this work, both authors were partially supported by the Royal Society University Research Fellowship Renewal 2022 URF\textbackslash R\textbackslash 221030. Part of this work was carried out during a research visit funded by Sapienza research project ``Algebraic and differential aspects of varieties and moduli spaces'' and LF wishes to thank the PI Gabriele Mondello for the support. FT was partially supported by the Pacific Institute for the Mathematical Sciences (PIMS), and would like to thank Jason Lotay for introducing him to the study of minimal surfaces in hyperk\"ahler 4-manifolds arising from the Gibbons--Hawking Ansatz. LF would like to thank Xuwen Zhu for mentioning \cite{WeakGravity} in connection to the construction of minimal 2-spheres in \cite{Foscolo2019}.

\section{The ambient space}\label{sec:Gibbons:Hawking}

In this section, we describe the simplest hyperk\"ahler $4$-manifolds that will contain the minimal surfaces that we aim to construct. Being symmetric $4$-pointed multi-Taub--NUT spaces, these metrics are given explicitly in terms of the celebrated Gibbons--Hawking ansatz \cite{GibbonsHawking1978}, which we recall below.

\subsection{The ansatz} Let $U\subset\R^3$ be an open subset and let $\pi:P\to U$ be a principal $S^1$-bundle over $U$ endowed with a connection 1-form $\theta$. Suppose that there exists a positive harmonic function $\phi$ on $U$ satisfying the abelian monopole equation:
\begin{align}\label{eqn: monopole equation}
    d\theta=\ast_{\R^3} d\phi,
\end{align}
i.e., $\ast d\phi$ is the curvature of the connection $\theta$. We endow $P$ with the explicit Riemannian metric:
\[
g_{gh}=\phi\,\pi^\ast g_{\R^3} +\phi^{-1}\,\theta^2,
\]
which induces the explicit hyperk\"ahler structure:
\begin{equation}\label{eqn: hyperkahler structure GH}
\omega_i^{gh}=dx_i\wedge \theta+\phi\, dx_j\wedge dx_k,
\end{equation}
where $x=(x_1,x_2,x_3)$ are fixed coordinates on $U$ and $(ijk)$ is any cyclic permutation of $(123)$. Using these formulas, it is clear that the twistor sphere can be identified with the unit sphere $S^2$ in the base $\R^3$. The principal $S^1$-action on $P$ is tri-holomorphic and therefore in particular isometric. 

In this work, we are interested in the examples where $U=\R^3\setminus\bigcup_{i=1}^n \{p_i\}$, for a fixed set of distinct points $p_1,...,p_n$, and where
\[
\phi=l^{-1}+\sum_{i=1}^n \frac{1}{2\abs{x-p_i}_{\R^3}}
\]
for a given constant $l>0$.
Such a set and harmonic function determine, up to gauge, a unique principal bundle $(P,\theta)$ over $U$. The metric defined by the ansatz then extends smoothly across each puncture by adding a single point. These spaces are complete with ALF (asymptotically locally flat) asymptotic geometry at infinity and are usually called multi-Taub--NUT spaces since, when there is only one characterizing point, the metric coincides with the classical Taub--NUT space.

In the rest of the paper, for every positive $d>0$, we let $(X_d, g_d)$ denote the multi-Taub--NUT space as above with $l=1$ and four characterizing points:
\begin{equation}\label{eqn: characterizing points Xd}
    p_1=\tfrac{\sqrt{2}}{2}(d,d,0), \qquad p_2=\tfrac{\sqrt{2}}{2}(d,-d,0), \qquad p_{-1}=\tfrac{\sqrt{2}}{2}(-d,-d,0), \qquad p_{-2}=\tfrac{\sqrt{2}}{2}(-d,d,0).
\end{equation}
Note that the normalisation $l=1$ can always be achieved by scaling.

\subsection{Minimal surfaces} We now recall some known facts about minimal surfaces in the spaces constructed via the Gibbons--Hawking ansatz. 

Lotay and Oliveira \cite{LotayOliveira2024}*{Section 4} classified all $S^1$--invariant minimal surfaces in these spaces. Indeed, they showed that any connected $S^1$--invariant minimal surface corresponds to a straight line, half-line or segment in $U\subset\R^3$. Any such minimal surface is calibrated by $\omega^{gh}_{\Dot{\gamma}}=\sum_{i=1}^3 \Dot{\gamma}_i\omega_i^{gh}$, where $\gamma=(\gamma_1,\gamma_2,\gamma_3)$ is the $g_{\R^3}$-arclength parametrisation of the projection to $\R^3$.  For example, in the Taub--NUT space, \emph{i.e.} the hyperk\"ahler space arising from the Gibbons--Hawking ansatz with only one characterizing point and $l>0$, for any fixed complex structure there exists a unique $S^1$--invariant complex curve biholomorphic to $\C$. This surface, often referred to as the ``cigar'', corresponds to the half-line passing through the characterizing point and with direction the one corresponding to the given complex structure. The cigar is one of the building blocks of the minimal surfaces we construct in this paper.

Another consequence of \cite{LotayOliveira2024} is that all compact $S^1$--invariant minimal surfaces of the multi-Taub--NUT spaces are diffeomorphic to $S^2$ and correspond to straight line segments connecting characterizing points. Moreover, it is easy to see that these minimal spheres are the unique homological area minimisers in their homology class (see f.i. \cite{Trinca2022}*{Corollary 3.34}) and, when there are at most two characterising points, these are the only compact minimal surfaces \cite{Trinca2022}*{Theorem 1.1}. In particular, in $X_d$ we have two minimal spheres corresponding to the straight line segments  $[p_{-1},p_{1}]$ and  $[p_{-2},p_{2}]$ intersecting along the circle $\pi^{-1}(0)$. The minimal surfaces we will construct can be thought of as desingularisations of this configuration.

\subsection{Symmetries}\label{sec: symmetries Xd} We now introduce a subgroup $\G<\Iso(X_d,g_d)$, which is isomorphic to the direct product of the dihedral group $\D_4$ with $\Z_2$ and will play a crucial role in our construction, see \cref{sec: Scherk's des} and \cref{sec: perturbation to minimal surface} below. The generators of this subgroup are certain isometric, anti-equivariant lifts to $X_d$ of the following isometries of $\R^3$:
\begin{equation*}
    \begin{aligned}
         \Tilde{R}_1(x_1,x_2,x_3)=(-x_1,x_2,x_3), \quad \Tilde{R}_2(x_1,x_2,x_3)=(x_1,-x_2,x_3), \quad \Tilde{R}_3(x_1,x_2,x_3)=(x_1,x_2,-x_3),\\
    \Tilde{R}_4(x_1,x_2,x_3)=(x_2,x_1,x_3), \qquad \Tilde{R}_5(x_1,x_2,x_3)=(-x_2,-x_1,x_3).\qquad\qquad \qquad
    \end{aligned}
\end{equation*}
The latter generate the group, isomorphic to $\D_4\times\Z_2$, of isometries of Euclidean $\R^3$ which preserve our choice of characterising points, hence the harmonic function $\phi$ and the origin. We will now show that the $\Tilde{R}_i$'s lift to anti-equivariant isometries of $(X_d, g_d)$, \emph{i.e.} isometries that intertwine the action of $e^{i\theta}\in S^1$ and $e^{-i\theta}$, satisfying the same relations as the $\Tilde{R}_i$'s. 


\begin{remark*}
    While any diffeomorphism of $\R^3$ preserving the characterising points lifts to $X_d$, in general it is not obvious that a group $\G$ of isometries of $\R^3$ preserving the characterising points has a lifted action on the corresponding Gibbons--Hawking space as it might be necessary to pass to a non-trivial extension of $\G$: for example, the antipodal map on $\R^3$ can only be lifted to an order-4 symmetry of the Taub--NUT space. 
\end{remark*}

In order to show that $\D_4\times\Z_2$ acts on $X_d$ lifting the action on $\R^3$, we recall that LeBrun \cite{LeBrun1991} proved that, in the complex structure corresponding to the direction $(0,0,1)\in \R^3$, $X_d$ is $S^1$--equivariantly (actually, $\C^*$--equivariantly) biholomorphic to the affine surface $D_d\subset \C^3$ defined by
\[
D_d=\left\{(w_1,w_2,w_3)\in \C^3: w_1 w_2=(w_3-\pi_\C(p_1))(w_3-\pi_\C(p_2))(w_3-\pi_\C(p_{-1}))(w_3-\pi_\C(p_{-2}))\right\},
\]
where $\pi_\C:\R^3\to\C$ is defined by $\pi_\C(x_1,x_2,x_3)=x_1+ix_2$ and $S^1$ acts on $\C^3$ as $e^{i\theta}\cdot (w_1, w_2, w_3) = (e^{i\theta}w_1, e^{-i\theta}w_2, w_3)$.
Given this identification, we define $\Bar{R}_i$, $i=1,\dots, 5$, to be the (anti)biholomorphisms of $\C^3$ defined by
\[
\Bar{R}_1(w_1,w_2,w_3)=(\bar{w}_1,\bar{w}_2,-\bar{w}_3), \quad \Bar{R}_2(w_1,w_2,w_3)=(\bar{w}_1,\bar{w}_2,\bar{w}_3), \quad \Bar{R}_3(w_1,w_2,w_3)=(w_2,w_1,w_3),
\]
\[
\Bar{R}_4(w_1,w_2,w_3)=(-\bar{w}_1,-\bar{w}_2,i\bar{w}_3), \qquad \Bar{R}_5(w_1,w_2,w_3)=(-\bar{w}_1,-\bar{w}_2,-i\bar{w}_3).
\]
They descend to anti-equivariant diffeomorphisms of $D_d$ and, using more explicitly the construction of LeBrun, one can check that the $\Bar{R}_i$ induce the given $\Tilde{R}_i$ on $\R^3$. For example, this is clear for the induced action on $\C\times \{ 0\}\subset \R^3$ since $w_3=x_1+ix_2$ in LeBrun's description. Note that the group generated by the $\Bar{R}_i$'s is again a copy of $D_4\times \Z_2$ since the relations
\[
\Bar{R}_1^2 = \Bar{R}_3^2=\Bar{R}_4^2=(\Bar{R}_1 \Bar{R}_3)^2 = (\Bar{R}_4 \Bar{R}_3)^2 = (\Bar{R}_1 \Bar{R}_4)^4 =1
\]
are easily checked.

Observe that, because of their anti-equivariance, we are free to compose all the symmetries $\Bar{R}_i$'s with the same element in the circle of triholomorphic isometries. In other words, our definition of the action of $\G$ on $X_d$ depends on the choice of a base point $p$ in the circle fibre $\pi^{-1}(0)$ over the origin in $\R^3$, which from now on we assume given.

It remains to show that the $\Bar{R}_i$'s act by isometries on $(X_d, g_d)$. Consider first the limiting singular case $d=0$, corresponding to the quotient of the Taub--NUT metric by $\Z_4$. The metric $g_0$ on $X_0$ can be realised as a hyperk\"ahler quotient of $\mathbb{H}\times (\R^3\times S^1)$ by the circle action $e^{i\theta}\cdot (q, x, e^{it}) = (e^{i\theta}q, x, e^{4i\theta}e^{it})$. This means we restrict to the level set $\{\overline{q}iq + 4x=0\}\subset \mathbb{H}\times \R^3\times S^1$ of the hyperk\"ahler moment map and quotient by the $S^1$--action. Writing $q=z_1 + j z_2$ with $z_1, z_2\in \C$ the identification with $D_0$ is given by setting $w_1 = z_1^4$, $w_2 = z_2^4$ and $w_3 = z_1 z_2$. The triholomorphic circle action on $X_0$ is induced by the circle action on the last $S^1$--factor and the corresponding hyperk\"ahler moment map is induced by the projection to $\R^3$.

Now, define isometries $\check{R}_1, \check{R}_3, \check{R}_4$ of $\mathbb{H}\times \R^3\times S^1$ by
\[
\check{R}_1(z_1, z_2, x, e^{it}) = (i\Bar{z}_1, i\Bar{z}_2, \Tilde{R}_1(x), e^{-it}), \qquad \check{R}_3(z_1, z_2, x, e^{it}) = (z_2, z_1, \Tilde{R}_3(x), e^{-it}),\]
\[
\check{R}_4(z_1, z_2, x, e^{it}) = (\zeta\Bar{z}_1, \zeta\Bar{z}_2, \Tilde{R}_4(x), -e^{-it}),
\]
where $\zeta$ is a unit complex number such that $\zeta^2=i$ (note the choice of $\zeta$ has no significance once we pass to the quotient by the $S^1$--action). It is clear that $\check{R}_i$ preserves the zero level set of the moment map, is anti-invariant with respect to the circle action on the last factor which induces the triholomorphic circle action on $X_0$, and acts as $\Bar{R}_i$ on $D_0$ and as $\Tilde{R}_i$ on the image $\R^3$ of the hyperk\"ahler moment map on $X_0$. Moreover, one checks that
\[
\check{R}_1^2 = \check{R}_3^2 = \check{R}_4^2 = (\check{R}_1 \check{R}_3)^2=(\check{R}_4 \check{R}_3)^2=1, \quad (\check{R}_1 \check{R}_4)^4(z_1, z_2, x, e^{it}) = (-z_1, -z_2, x, e^{it})
\sim (z_1, z_2, x, e^{it})
\]
modulo the action of $S^1$. We conclude that $D_4\times \Z_2$ acts by isometries on $g_0$ as claimed.

From this, one can argue that the action of $\G$ on $X_d$ is also by isometries, since the metric $g_d$ is the unique ALF K\"ahler Ricci-flat metric on the complex affine variety $D_d$ with exact K\"ahler form and which is asymptotic to the quotient of the Taub--NUT metric by $\Z_4$. Such a uniqueness result is a consequence of the classification of ALF gravitational instantons of cyclic type by Minerbe \cite{Minerbe}.
Alternatively (and more explicitly), we can also realise the metric $g_d$ as a hyperk\"ahler quotient of $\mathbb{H}^4\times \R^3\times S^1$ by $T^4$, where the $j$th factor $S^1\subset T^4$ acts diagonally on the corresponding factor $\mathbb{H}\subset\mathbb{H}^4$ and on $\R^3\times S^1$. As for the $d=0$ case, we can define a diffeomorphism $\check{R}_i$ of $\mathbb{H}^4\times \R^3\times S^1$ composing the definition of $\check{R}_i$ for $d=0$ with the permutation of the four $\mathbb{H}$--factors in $\mathbb{H}^4$ corresponding to the permutation of the characterising points $\{p_{\pm 1}, p_{\pm 2}\}$ induced by $\Tilde{R}_i$. The $\check{R}_i$'s then induce an isometric action of $\G$ on the hyperk\"ahler quotient $X_d$, which coincides with the action given by the $\Bar{R}_i$'s.

Note that in the hyperk\"ahler quotient description of $X_d$ the action of $\G$ on the factor $\R^3\times S^1$ is given by
\begin{equation}\label{eqn: G in R3xS1}
    \begin{aligned}
    R_i\left((x_1,x_2,x_3),e^{it}\right)&=\left(\Tilde{R}_i(x_1,x_2,x_3), e^{-it}\right),\\
    R_j\left((x_1,x_2,x_3),e^{it}\right)&=\left(\Tilde{R}_j(x_1,x_2,x_3), -e^{-it}\right)
    \end{aligned}
\end{equation}
for $i=1,2,3$ and $j=4,5$. The projection $\mathbb{H}^4\times \R^3\times S^1\rightarrow \R^3\times S^1$ induces a $\G$--equivariant diffeomorphism between $\pi^{-1}(\Omega)\subset X_d$ and $\Omega\times S^1$ for any $\G$--invariant open subset $\Omega\subset \R^3\setminus \{ p_1, p_{2}, p_{-1}, p_{-2}\}$. This can be shown explicitly from the hyperk\"ahler quotient description observing that for every $x\neq p_i$ the equation $\overline{q}iq + x = p_i$ for $q\in\mathbb{H}$ has solutions given by a non-trivial $S^1$--orbit. Alternatively, we have the following

\begin{lemma}[Equivariant trivialization near the origin]\label{lemma: exp gauge}
    For every $T<d$, there exists a $\G$--equivariant trivialization of $\pi^{-1}(\EBall{T})\subset X^d$. Moreover, the connection 1-form $\theta$ has no radial component in this trivialization. 
\end{lemma}
\begin{proof}
Consider the trivialization induced by the section of the $S^1$--bundle which coincides with $p$ at $\pi^{-1}(0)$ and is extended to $\EBall{T}$ via parallel transport along Euclidean radial geodesics from the origin. Since the $\Bar{R}_i$'s preserve the connection 1-form $\theta$ up to sign and the $\Tilde{R}_i$'s map radial geodesics from the origin to radial geodesics from the origin, the constructed section is $\G$--equivariant and the lemma follows. 
\end{proof}

\begin{remark}\label{rmK:Symmetries:hk:triple}
The induced action of $D_4\times \Z_2$ on the triple of hyperk\"ahler forms on $X_d$ is generated by $-\Tilde{R}_i$, $i=1,\dots, 5$, as can be checked explicitly from \eqref{eqn: hyperkahler structure GH}. In particular, the induced action of $D_4\times \Z_2$ on the copy of $\R^3$ spanned by the hyperk\"ahler triple has no non-zero fixed points.  
\end{remark}

\subsection{Metric estimates} In this subsection, we prove that for $d$ large enough, the metric $g_d$ on $X_d$ is close to the flat one in a neighbourhood of $\pi^{-1}(0)\subset X_d$, and to the Taub--NUT metric in a neighbourhood of $\pi^{-1}(p_i)$ for every $i=\pm1,\pm2$.

\begin{lemma}\label{lemma: estimates phi near 0} For $\abs{x}\ll d$ we have
\[
\phi=m+\bigO\left(\frac{\abs{x}}{d^2}\right), \qquad  \nabla^k \phi=\bigO\left(d^{-k-1}\right)
\]
for all $k\geq 1$, where $m:=1+\frac{2}{d}$ and $\nabla$, $\abs{\,\cdot\,}$ are calculated using the Euclidean metric $g_{\R^3}$.
\end{lemma}
\begin{proof}
    By writing
    \[
    \phi=1+\sum_{i=1}^4 \frac{1}{2\abs{x-p_i}}=1+\frac{1}{2d}\sum_{i=1}^4 \frac{1}{\abs{\frac{x}{d}-\frac{p_i}{d}}},
    \]
    it is clear that the estimates follow by Taylor's Theorem for $d^{-1}x \approx 0$.
\end{proof}
\begin{lemma}\label{lemma: estimates theta near 0} In the trivialization of \cref{lemma: exp gauge}, for $\abs{x}\ll d$ we have
    \[
    \theta=dt + \bigO\left(\frac{\abs{x}}{d^2}\right), \qquad \nabla^k\theta=\bigO\left(d^{-k-1}\right),
    \]
for all $k\geq 1$, where $\nabla$, $\abs{\,\cdot\,}$ are calculated using the Euclidean metric $g_{\R^3}$.
\end{lemma}
\begin{proof}
In any local trivialization around $0\in\R^3$, the connection $1$-form $\theta$ can be written as follows:
\[
\theta=dt+a_r dr+a_\alpha d\alpha +a_\beta d\beta,
\]
where $r\in[0,\infty)$, $\alpha\in[0,\pi)$, $\beta\in[0,2\pi)$ are the standard spherical coordinates on $\R^3$. By our choice of trivialization, we have that $a_r=0$. Hence, the monopole equation, \cref{eqn: monopole equation}, implies:
\[
\left\lvert\DerParz{a_\alpha}{r}\right\rvert\leq r\abs{\nabla\phi}, \qquad \left\lvert\DerParz{a_\beta}{r}\right\rvert\leq r\sin\alpha\abs{\nabla\phi}.
\]
From this equation and \cref{lemma: estimates phi near 0}, we deduce that:
\[
\abs{a_\alpha d\alpha}\leq C r/d^2, \qquad \abs{a_\beta d\beta}\leq C r/d^2,
\]
and, hence, the desired estimates.
\end{proof}
From \cref{lemma: estimates phi near 0} and \cref{lemma: estimates theta near 0}, we obtain the following estimates for the metric $g_d$.

\begin{lemma}\label{lemma: metric estimates near 0} In the trivialization of \cref{lemma: exp gauge}, for $\abs{x}\ll d$ we have
\[
g_d=g_m+\bigO\left(\frac{\abs{x}}{d^2}\right) \qquad \nabla^k g_d=\bigO\left(d^{-k-1}\right)
\]
for all $k\geq 1$, where  $\nabla$, $\abs{\,.\,}$ are computed with respect to the flat metric $g_m:=m\, g_{\R^3}+m^{-1}\, dt^2$, $m=1+2/d$.
\end{lemma}

In a neighbourhood of a puncture, it is clear that the metric approaches the Taub--NUT metric. In this region, we do not need precise quantitative estimates.

\begin{lemma}\label{lemma: metric estimates near punctures} 
For any $T>0$ let $\pi_{\TN}^{-1}(\EBall{T})$ be a local chart of the Taub--NUT space, i.e., the space constructed via the Gibbons--Hawking ansatz using the harmonic function $\phi_{\TN}=1+\frac{1}{2\abs{x}}$.
Then for any $i=\pm1,\pm2$ and for $d$ large enough, there exists an $S^1$--equivariant diffeomorphism $f_i\co \pi_{\TN}^{-1}(\EBall{T}) \rightarrow \pi^{-1}(\EBallc{T}{p_i}) \subset X_d$ covering the linear map $x\mapsto x+p_i$ on $\R^3$ and such that $f_i^\ast g_d$ smoothly converges to the Taub--NUT metric $g_{\TN}$ as $d\to\infty$. 
\end{lemma}

\section{Approximate minimal surfaces}\label{sec: aprroximate minimal surface}
In this section, we construct initial approximate minimal surfaces in $(X_d,g_d)$ for every $d$ sufficiently large. The idea is to desingularise the intersection of the two $S^1$--invariant minimal spheres corresponding to the straight line segments $[p_{-1},p_{1}]$ and $[p_{-2},p_{2}]$ by gluing a singly periodic Scherk's surface, thought of as a minimal surface in $\R^2\times S^1\subset \R^3\times S^1$ by quotienting by the period, near the intersecting circle. As a matter of fact, it is more accurate to think of this gluing construction in terms of the four $S^1$--invariant holomorphic discs corresponding to $[0,p_i]$, which are holomorphic with respect to the four distinct complex structures $p_i/|p_i|$ for $i=\pm 1,\pm2$ and that intersect in the common boundary $\pi^{-1}(0)$. Each end of Scherk's surface, which is asymptotic to a cylinder after we quotient by the period, is glued to one of these $S^1$--invariant holomorphic curves. Since Scherk's surface is not $S^1$--invariant, we have in fact a circle of approximate minimal surfaces obtained in this way.

\subsection{Scherk's desingularization}\label{sec: Scherk's des} Scherk's singly periodic surface $\Sc$ (of parameter $\pi/4$) is classically known as the minimal surface in $\R^3$ defined by the equation
\begin{equation}\label{eqn: eq Scherk}
    \sinh{x_1}\cosh{x_2}=\sin{t}.
\end{equation}
Because $\Sc$ is $2\pi$--periodic in $t$, it can be regarded as a minimal surface in $\R^2\times \R/2\pi\Z\cong\R^2\times S^1$. 

We now recall some important geometric properties of $\Sc\subset \R^2\times S^1$ as studied by Kapouleas in his gluing constructions (cfr. \cites{Kapouleas97, Kapouleas2011}). First of all, $\Sc$ is invariant under the following isometries of $\R^2\times S^1$:
\[
(x_1,x_2,e^{it})\mapsto(-x_1,x_2,e^{it}), \qquad (x_1,x_2,e^{it})\mapsto(x_1,-x_2,e^{it}), \qquad (x_1,x_2,e^{it})\mapsto(x_1,x_2,e^{-it}),
\]
\[
(x_1,x_2,e^{it})\mapsto(x_2,x_1,-e^{-it}), \qquad (x_1,x_2,x_3)\mapsto(-x_2,-x_1,-e^{-it}).
\]
Secondly, $\Sc$ has four exponentially asymptotically cylindrical ends. The asymptotic cylinders are the $S^1$--invariant submanifolds
\[
\left\{(x,x,e^{it})\in\R^2\times S^1: x\in\R^\pm, \,  t\in[0,2\pi)\right\}, \qquad \left\{(x,-x,e^{it})\in\R^2\times S^1: x\in\R^\pm,  \, t\in[0,2\pi)\right\}.
\]
Finally, the Gauss map is a diffeomorphism from $\Sc$ into $S^2$ with four punctures. The punctures corresponds to the four normals of the asymptotic cylinders. Embedding $\R^2\times S^1$ in $\R^3\times S^1$ as the totally geodesic hypersurface $\{ x_3=0\}$ we regard $\Sc$ as a minimal surface in $\R^3\times S^1$. In fact, in view of \cref{lemma: metric estimates near 0} we rescale by $m=1+2d^{-1}\approx 1$ and regard $\Sc$ as a minimal surface in $(\R^3\times S^1, g_m)$.

For $0<2T<d$, the trivialisation of $\pi^{-1}(\EBall{2T})\subset X_d$ given by \cref{lemma: exp gauge} yields a natural way to embed a piece of $\Sc$ in $\pi^{-1}(\EBall{2T})$, which we denote by $\Sc_T$. In fact, given the trivialisation of \cref{lemma: exp gauge} depends on a choice of base-point $p\in \pi^{-1}(0)$, we have an $S^1$--family of surfaces $\Sc_T$ in $\pi^{-1}(\EBall{T})$. We fix a choice of base point so that we work with a single surface $\Sc_T$ and a group $\G\simeq D_4\times \Z_2$ acting on $X_d$ by isometries.
Most properties of $\Sc\subset \R^2\times S^1$ still hold for $\Sc_T$. We enumerate them in the following proposition, which is a straightforward consequence of \cite{Kapouleas2011}*{Section 5}.

\begin{proposition}\label{prop: properties of Scherk}
    There exists $d_0>0$ such that the following holds for every $d>d_0$. There exist $0<t_0<T_0$ such that for any $T_0<T<\frac{1}{2}d$ we have: 
    \begin{enumerate}
        \item $\Sc_T$ is invariant under $\G$;
        \item $\Sc_T \cap \pi^{-1}\left( \EBall{T}\setminus \EBall{t_0}\right)$ has four connected components labelled by $i=\pm 1, \pm 2$. The $i$th component is the graph of a normal vector field $\nu_i$ on the cyclinder $C^i$ parametrised by
        \[
        (t, e^{i\phi}) \longmapsto \left(t\frac{p_i}{d},e^{i\phi}\right).
        \]
        The normal vector fields $\nu_i$'s are given by
        \[
        \nu_i(t, e^{i\phi})=f(t,e^{i\phi})\, \widehat{p}_i^{\perp},
        \]
        where $f:(t_0,\infty)\times S^1 \to\R$ is a smooth exponentially decaying function independent of $i$ and
        \begin{align*}
             & \widehat{p}_1^{\perp}=\tfrac{\sqrt{2}}{2}(-\partial_{x_1}+\partial_{x_2}), \quad  \widehat{p}_2^{\perp}=\tfrac{\sqrt{2}}{2}(-\partial_{x_1}-\partial_{x_2}), \\
              & \widehat{p}_{-1}^{\perp}=\tfrac{\sqrt{2}}{2}(+\partial_{x_1}-\partial_{x_2}), \quad  \widehat{p}_{-2}^{\perp}=\tfrac{\sqrt{2}}{2}(+\partial_{x_1}+\partial_{x_2}).
        \end{align*}
        Here by abuse of notation $\partial_{x_i}$ denotes the horizontal lift of the corresponding vector field on $B^{\R^3}_T$ to $X_d$.
      \item For every $\epsilon>0$ the constant $T_0$ can be chosen so that the (positive) Gauss lift $a:\Sc_T\to S^2\setminus \bigcup_{i=\pm1,\pm2} B_\epsilon(p_i/d)$ is a diffeomorphism, where we identify the twistor sphere of $X_d$ with the unit sphere in $\R^3$.
    \end{enumerate}
\end{proposition}

\begin{remark}
Note that suitable portions of the asymptotic cylinders $C^{i}$ coincide with pieces of the $S^1$--invariant minimal discs corresponding to the segments $[0,p_i]$, which are holomorphic with respect to the complex structure corresponding to $p_i/d$. 
\end{remark}

\subsection{Construction of the initial surface}\label{sec: construction initial surface} We are now ready to provide the construction of the initial surface in $X_d$. Let $T_1<T_2<T<d$ be large enough so that \cref{prop: properties of Scherk} applies and let $\chi_{T_1,T_2}$ be a cut-off function such that
\[
\chi_{T_1,T_2}(t)=\begin{cases} 
      1 & t\leq T_1 \\
      0 & t\geq T_2 
   \end{cases},
\qquad \abs{\nabla\chi_{T_1,T_2}} + \abs{\nabla^2\chi_{T_1,T_2}} \leq \frac{C}{T_2-T_1},
\]
for some positive constant $C$ independent of $T_1,T_2$. We then define $\ScIn$ as the surface in $X_d$ given as follows:
\begin{itemize}
    \item In $\pi^{-1}(\EBall{T_1})$, $\ScIn$ coincides with $\Sc_{T}$, as given in \cref{sec: Scherk's des}. This is the green region in \cref{fig: proj initial surface}. The model for this piece of $\ScIn$ is the Scherk's surface naturally embedded in $\R^2\times\{0\}\times S^1\subset\R^3\times S^1$ (cfr. \cref{sec: linear analysis on Scherk} below).

    \item In $\pi^{-1}(\EBall{T_2}\setminus \EBall{T_1})$, $\Sc_T$ is the union of the graphs of $\nu_i$, as defined in \cref{prop: properties of Scherk}, for $i=\pm1,\pm2$. We define $\ScIn$ as the graph of:
   \[
        \chi\nu_i|_{ b_i(t,e^{i\phi})}:=\chi_{T_1,T_2}(t)\,  f(t,e^{i\phi})\, \widehat{p}_i^{\perp}.
    \]
    This is the orange region in \cref{fig: proj initial surface}. The model for this piece of $\ScIn$ is a totally geodesic cylinder $\R\times\{\underline{0}\}\times S^1$ in $(\R^3\times S^1, dt^2+g_{\R^3})$ (cfr. \cref{sec: linear analysis on cylinder} below).
    
    \item Outside of $\pi^{-1}(\EBall{T_2})$, $\ScIn$ consists of four pieces. Each piece is the $S^1$-invariant minimal surfaces corresponding to the segment $[T_2\frac{p_i}{d},p_i]$, for $i=\pm1,\pm2$. This is the blue region in \cref{fig: proj initial surface}. The model for this piece of $\ScIn$ is the cigar in the Taub--NUT space (cfr. \cref{sec: linear analysis on cigar} below).
\end{itemize}
\begin{remark}\label{remark: properties ScIn}
    It is clear that $\ScIn$ is a $\G$-invariant smooth submanifold of $X_d$, diffeomorphic to $S^2$ and with degree-1 (positive) Gauss lift. Moreover, it is also clear that the construction is independent from the choice of $T$ and only depends on $T_1, T_2$ as the notation suggests. 
\end{remark}

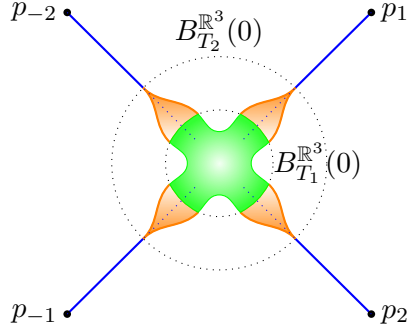
\begin{figure}
    \centering
 \begin{tikzpicture}

\draw[black, dotted] (0,0) circle (1.41421356237 cm);
\draw[black, dotted] (0,0) circle (0.70710678118 cm);

\draw[blue, thick] (2,2) to (1,1);
\draw[blue, thick] (-2,-2) to (-1,-1);
\draw[blue, thick] (2,-2) to (1,-1);
\draw[blue, thick] (-2,2) to (-1,1);

\draw[orange, shade, top color=orange!90] (1,1) to[out=225,in=45] (0.65,0.278) to[out=120, in=330] (0.278,0.65) to[out=45,in=225] (1,1);
\draw[orange, shade, top color=orange!90, rotate=90] (1,1) to[out=225,in=45] (0.65,0.278) to[out=120, in=330] (0.278,0.65) to[out=45,in=225] (1,1);
\draw[orange, shade, top color=orange!5, bottom color=orange!90, rotate=180] (1,1) to[out=225,in=45] (0.65,0.278) to[out=120, in=330] (0.278,0.65) to[out=45,in=225] (1,1);
\draw[orange, shade, top color=orange!5, bottom color=orange!90, rotate=270] (1,1) to[out=225,in=45] (0.65,0.278) to[out=120, in=330] (0.278,0.65) to[out=45,in=225] (1,1);

\draw[orange, thick] (1,1) to[out=225,in=45] (0.65,0.278);
\draw[orange, thick] (1,1) to[out=225,in=45] (0.278,0.65);
\draw[orange, thick, rotate=90] (1,1) to[out=225,in=45] (0.65,0.278);
\draw[orange, thick, rotate=90] (1,1) to[out=225,in=45] (0.278,0.65);
\draw[orange, thick, rotate=180] (1,1) to[out=225,in=45] (0.65,0.278);
\draw[orange, thick, rotate=180] (1,1) to[out=225,in=45] (0.278,0.65);
\draw[orange, thick, rotate=270] (1,1) to[out=225,in=45] (0.65,0.278);
\draw[orange, thick, rotate=270] (1,1) to[out=225,in=45] (0.278,0.65);

\draw[green, shade, outer color=green!80, inner color=green!5] (0.278,0.65) to[out=330, in=120]  (0.65,0.278) to[out=225,in=90] (0.42,0) to[out=270, in=135] (0.65,-0.278) to[out=240, in=30] (0.278,-0.65) to[out=135,in=0] (0,-0.42) to[out=180, in=45] (-0.278,-0.65) to[out=150, in=300] (-0.65,-0.278) to[out=45,in=270] (-0.42,0) to[out=90, in=315] (-0.65,0.278) to[out=60, in=210] (-0.278,0.65) to[out=315,in=180] (0,0.42) to[out=0, in=225] (0.278,0.65);

\node at (1.3,0) [] {$B^{\mathbb{R}^3}_{T_1}(0)$};
\node at (0,1.75) [] {$B^{\mathbb{R}^3}_{T_2}(0)$};

\fill[black] (2,2) circle (0.05cm) node[black,right] {$p_1$};
\fill[black] (-2,-2) circle (0.05cm) node[black,left] {$p_{-1}$};
\fill[black] (2,-2) circle (0.05cm) node[black,right] {$p_{2}$};
\fill[black] (-2,2) circle (0.05cm) node[black,left] {$p_{-2}$};

\draw[blue, dotted] (-2,2) to (-0.3,0.3);
\draw[blue, dotted] (-2,-2) to (-0.3,-0.3);
\draw[blue, dotted] (2,-2) to (0.3,-0.3);
\draw[blue, dotted] (2,2) to (0.3,0.3);
\end{tikzpicture}

    \caption{Schematic representation of the initial surface $\ScIn$ in $\R^3$.}
    \label{fig: proj initial surface}
\end{figure}

\subsection{Mean curvature estimates}\label{sec: mean curvature estiamte} In this subsection, we estimate the mean curvature of the initial surface $\ScIn$ constructed above. Since outside of $\pi^{-1}(\EBall{T_2})$ it coincides with four complex submanifolds (w.r.t. 4 distinct complex structures compatible with $g_d$), it is enough to estimate the mean curvature of $\ScIn$ inside of $\pi^{-1}(\EBall{T_1})$ and in the annulus $\pi^{-1}(\EBall{T_2}\setminus\EBall{T_1})$. We denote by $H^d_{T_1,T_2}$ the mean curvature of $\ScIn\subset(X_d,g_d)$. 

We will make extensive use of the classical fact that the mean curvature operator smoothly depends on the metric and on small perturbations of the given submanifold (cfr. \cite{White1991}*{Theorem 1.1}). 

\begin{lemma}\label{lemma: estimate mean curvature Sc} For $d$ big enough, in the region $\pi^{-1}(\EBall{T_1})$ we have:
\[
\bigabs{H^d_{T_1,T_2}}_{\Czeroa}=\bigO\left(\frac{T_1}{d^2}\right),
\]
where $\Czeroa=\Czeroa(N\ScIn)$.
\end{lemma}
\begin{proof}
    By the triangle inequality, the smoothness of the mean curvature operator with respect to the metric and \cref{lemma: metric estimates near 0} we have:
    \[
    \bigabs{H^d_{T_1,T_2}}_{\Czeroa}\leq \bigabs{H^d_{T_1,T_2}-H^m_{T_1,T_2}}_{\Czeroa}+\bigabs{H^m_{T_1,T_2}}_{\Czeroa}=\bigabs{H^d_{T_1,T_2}-H^m_{T_1,T_2}}_{\Czeroa}\leq C\frac{T_1}{d^2},
    \]
    where $H^m_{T_1,T_2}$ is the mean curvature of $\ScIn$ with respect to the flat metric $g_m$ defined in \cref{lemma: metric estimates near 0}. Here we also used the fact that $H^m_{T_1,T_2}=0$ because, in this region, $\ScIn$ coincides with the standard Scherk surface, $\Sc$, which is minimal with respect to the flat metric $g_m$, and the metric estimates .    
\end{proof}

\begin{lemma}\label{lemma: estmate mean curvature neck}
    For $d$ big enough, in the region $\pi^{-1}(\EBall{T_2}\setminus \EBall{T_1})$ we have:
    \[
    \bigabs{H^d_{T_1,T_2}}_{\Czeroa}=\bigO\left(e^{-T_1}\left(\frac{T_2}{d^2}+\frac{1}{T_2-T_1}+e^{-T_1}\right)\right),
    \]
    where $\Czeroa=\Czeroa(N\ScIn)$.
\end{lemma}
\begin{proof} 
    By the triangle inequality we have:
    \begin{align}\label{eqn: lemma mean curvature aux}
        \bigabs{H^d_{T_1,T_2}}_{\Czeroa}\leq \bigabs{H^d_{T_1,T_2}-H^m_{T_1,T_2}}_{\Czeroa}+\bigabs{H^m_{T_1,T_2}}_{\Czeroa}, 
    \end{align}
    where $H^m_{T_1,T_2}$ is the mean curvature of $\ScIn$ with respect to the metric $g_m$ defined in \cref{lemma: metric estimates near 0}.

    Since the mean curvature operator is smooth, we can use \cref{lemma: metric estimates near 0} and \cref{prop: properties of Scherk} to estimate the first term as follows:
    \[
    \bigabs{H^d_{T_1,T_2}-H^m_{T_1,T_2}}_{\Czeroa}\leq C\abs{g_d-g_m}\abs{\chi_{T_1,T_2}\nu}_{\Ctwoa}\leq C\frac{T_2}{d^2}e^{-T_1}.
    \]
    Now, writing the mean curvature operator for a graph over the minimal cylinder $C^i$ with respect to the metric $g_m$ as the sum of a linear operator $\L_{C^i}$ and non-linear terms $Q_{C^i}$, we can control, for every $i$, the second term of \cref{eqn: lemma mean curvature aux} as follows: 
    \begin{align*}
    \bigabs{H^m_{T_1,T_2}}_{\Czeroa}=&\bigabs{\L_{C^i}(\chi_{T_1,T_2}\nu_i)+ Q_{C^i}(\chi_{T_1,T_2}\nu_i)}_{\Czeroa}\\
    \leq& \bigabs{\nabla^2\chi_{T_1,T_2}\ast \nu_i+\nabla\chi_{T_1,T_2}\ast \nabla\nu_i}\\
    &+\bigabs{\chi_{T_1,T_2}\,(\L_{C^i}\nu_i+Q_{C^i}\nu_i)}+\bigabs{Q_{C^i}(\chi_{T_1,T_2}\nu_i)-\chi_{T_1,T_2}\,Q_{C^i}\nu_i}_{\Czeroa}\\
    \leq& C\frac{1}{T_2-T_1} e^{-T_1}+0+Ce^{-2T_1}, 
     \end{align*}
     where we used \cref{prop: properties of Scherk} and once again the fact that the Scherk surface is minimal with respect to $g_m$, i.e., $(\L_{C^i}\nu_i+Q_{C^i}\nu_i)=0$.
\end{proof}
We now choose, once and for all, $T_1, T_2$ such that $T_2-T_1=1$ and $T_1=W_0(d^2)$, where $W_0$ is the real part of the principal branch of the Lambert $W$ function. Implicitly, $T_1$ is defined by the equation $T_1e^{T_1}=d^2$. Under this choice of $T_1,T_2$, we denote $\ScIn$ by $\Sc^d$. With these choices, \cref{lemma: estimate mean curvature Sc} and \cref{lemma: estmate mean curvature neck} imply the following global estimate for the mean curvature of $\Sc^d$. 

\begin{proposition}\label{prop: mean curvature estimates initial surface} For $d$ big enough, the mean curvature, $H^d$, of $\Sc^d\subset(X^d,g_d)$ is such that:
\[
 \bigabs{H^d}_{\Czeroa}=\bigO\left(\frac{\log d}{d^2}\right), 
\] 
where $\Czeroa=\Czeroa(N\Sc^d)$
\end{proposition}

\begin{remark*}
    Our choice of $T_1,T_2$ in terms of $d$ minimises the error $H^d$ in terms of the single parameter $d$.
\end{remark*}

\section{Perturbation to minimal surfaces}\label{sec: perturbation to minimal surface}   
In the previous section, for all $d$ large enough we have constructed an approximate minimal surface $\Sc^d\subset (X_d,g_d)$, in the sense that the mean curvature of $\Sc^d$ is arbitrarily small as $d\to+\infty$. We would like to deform $\Sc^d$ into an actual minimal surface using the Implicit Function Theorem. As usual in gluing constructions, this deformation needs some care, as $X_d$ degenerates as $d\to+\infty$.

The abstract argument goes as follows. Given $\Sc^d\subset (X_d,g_d)$ as above, we want to find a section $\nu$ of the normal bundle of $\Sc^d$, such that $\graph\nu$ in $X_d$ is minimal. The mean curvature of $\graph\nu$, denoted by $H_\nu$, satisfies the equation:
\[
H_\nu=H^d+\L \nu+ Q\nu,
\]
where $\L=\Delta^{\perp}+\mathcal{A}+\mathcal{R}$ is the linearized (Jacobi) operator, and $Q$ contains the non-linearities. By a quantitative version of the Implicit Function Theorem (see \cref{lemma: contraction map}), it is enough to find estimates uniform in $d$ for the inverse of $\L$ (see \cref{prop: uniform linear estimates}) and the non-linear term $Q$ (see \cref{prop: uniform nonlinearities estimates}) in suitable Banach spaces. Indeed, \cref{prop: mean curvature estimates initial surface} implies that $\Sc^d$ has mean curvature arbitrary small for $d\to\infty$, hence the uniform estimates mentioned above allow one to apply the quantitative Implicit Function Theorem \cref{lemma: contraction map} for any $d$ large enough to find a minimal surface close to $\Sc^d$.

In our case, the suitable Banach spaces are $\Cka_\G(N\Sc^d)=\left\{\nu\in\Cka(N\Sc^d):\text{$\nu$ is $\G$-equivariant} \right\}$ endowed with the obvious norm. Since $\Sc^d$ is $\G$-invariant, $H^d\in \Czeroa_\G(N\Sc^d)$ and therefore it makes sense to solve the equations in these spaces. 
\begin{remark*} Working with equivariant sections of the normal bundle is essential to prevent non-trivial bounded kernel and co-kernel of the linearized operator on the building blocks (extended substitute kernel and cokernel in the language of \cite{Kapouleas2011}). Without the symmetry assumption, the gluing problem becomes much more challenging because it is then necessary to introduce deformations of the initial surface $\Sc^d$ to compensate for the extended substitute cokernel. 
\end{remark*}

\subsection{The linear problem on the building blocks}\label{sec: linear problem on building blocks}
The building blocks for our gluing construction are the Scherk surface, as described in \cref{sec: Scherk's des}, a totally geodesic cylinder in $\R^3\times S^1$ and the cigar in the Taub--NUT space. Since we are working in a $\G$-equivariant setting, we will endow Scherk's surface, the totally geodesic cylinder and the cigar with natural induced actions.


\subsubsection{Scherk's surface}\label{sec: linear analysis on Scherk} Let $\Sc\subset \R^2\times\{0\}\times S^1\subset (\R^3\times S^1, dt^2+g_{\R^3})$ be the Scherk surface defined by \cref{eqn: eq Scherk}, let $\L_\Sc$ be the linearized operator of the mean curvature of $\Sc$ and let $\G$ be the subgroup of isometries of $(\R^3\times S^1, dt^2+g_{\R^3})$ generated by the maps given in \cref{eqn: G in R3xS1}. It is clear from \cref{sec: Scherk's des} that $\Sc$ is invariant under $\G$. 

\begin{proposition}\label{prop: linearized problem on Scherk}
    Let $\nu\in\Ctwoa_{loc}(N\Sc)$ be $\G$-equivariant, bounded and such that $\L_{\Sc}\nu=0$. Then $\nu\equiv0$.
\end{proposition}
\begin{proof}
    Since $\Sc$ is embedded in $\R^2\times\{0\}\times S^1$, the normal bundle of $\Sc$ can be trivialized by $(\nu_1,\nu_2)$, where $\nu_1$ represents the unit normal of $\Sc$ as an hypersurface of $\R^2\times\{0\}\times S^1$ and $\nu_2=\partial_{x_3}$. Because $\R^2\times\{0\}\times S^1$ is totally geodesic in $\R^3\times S^1$, the second fundamental form of $\Sc$, which we denote by $\II_\Sc$, has no component in the $\nu_2$ direction. It follows that in this trivialization the linearized operator has the form:
    \begin{equation}\label{eqn: linearized Scherk}
    \L_{\Sc}=\frac{\abs{\II_\Sc}^2}{2}(\Delta_{S^2}+2,\Delta_{S^2}),
    \end{equation}
    where $\Delta_{S^2}$ denotes the (analysts') Laplacian with respect to the round metric on $S^2$. Here we are using that the Gauss map of $\Sc\subset \R^2\times\{0\}\times S^1$ is a diffeomorphism into $S^2$ with four punctures \cites{Kapouleas97,Kapouleas2011} and the classical fact that the Gauss map of a minimal surface is conformal with conformal factor $\frac{\abs{\II_\Sc}^2}{2}$. 

    From \cref{eqn: linearized Scherk}, we can deduce that the bounded kernel of $\L_\Sc$ is the four-dimensional vector space spanned by: $X_1\eqdef\scal{\nu_1}{\partial_{x_1}}\nu_1$, $X_2\eqdef\scal{\nu_1}{\partial_{x_2}}\nu_1$, $X_3\eqdef\scal{\nu_1}{\partial_{t}}\nu_1$ and $X_4\eqdef\nu_2$ (cfr. \cite{MontielRos1991}*{Corollary 15}). Now, every element of the bounded kernel of $\L_\Sc$ that is also $\G$-equivariant needs to be identically zero. Indeed, $(R_i)_\ast X_i=-X_i$ for $i=1,2$, $(R_1)_\ast X_3=-X_3$ and $(R_3)_\ast X_4=-X_4$.
\end{proof}

We now explain how the Scherk surface can be regarded as a ``building block'' for the initial surface constructed in the previous section. Let $\Sc^d\subset X_d$ be the surface constructed in \cref{sec: aprroximate minimal surface}. For every fixed $R>0$, we can use \cref{lemma: exp gauge} to $\G$-equivariantly identify $\pi^{-1}(\EBall{R})$ with $\EBall{R}\times S^1$. In this trivialization, the piece of initial surface $\Sc^d\cap \pi^{-1}(\EBall{R})\subset X_d$ coincides $\G$-equivariantly with the piece of Scherk surface $\Sc\cap (\EBall{R}\times S^1)\subset\R^3\times S^1$ for every $d$ large enough (cfr. \cref{sec: construction initial surface}). The normal bundle of $\Sc^d\cap \pi^{-1}(\EBall{R})$ can be $\G$-equivariantly identified with the normal bundle of $\Sc\cap (\EBall{R}\times S^1)$. Moreover, since \cref{lemma: metric estimates near 0} implies that the metric $g_d$ smoothly converges to the flat metric $dt^2+g_{\R^3}$, we see that the linearized operator $\L_{\Sc^d}$ of $\Sc^d\cap \pi^{-1}(\EBall{R})$ smoothly converges, as $d\to+\infty$, to the linearized operator $\L_\Sc$ of the piece of Scherk surface $\Sc\cap (\EBall{R}\times S^1)$.


\subsubsection{Cigar in the Taub--NUT space}\label{sec: linear analysis on cigar} Let $L$ be the cigar in the Taub--NUT space and denote with $\L_L$ the Jacobi operator of $L$, i.e., the linearisation of the mean curavature operator for normal graphs over $L$. Recall that $L$ is the surface $\pi_{\TN}^{-1}(\{(x_1,0,0)\in\R^3: x_1\geq0\})$ in the space constructed via the Gibbons--Hawking ansatz with $\phi_{\TN}=l+\frac{1}{2\abs{x}}$ and $U=\R^3\setminus\{0\}$. It is well known that $L$ is holomorphic with respect to the complex structure associated to $\omega^{TN}_1=dx_1\wedge\theta+\phi_{\TN} dx_2\wedge dx_3$: it is in fact a coordinate line $\C\times\{0\}$ in the identification of $\TN$ with $\C^2$ \cite{LeBrun1991}. The name ``cigar'' is due to the fact that the induced metric is asymptotically cylindrical (though not with exponential decay).

\begin{proposition}\label{prop: linearized problem on cigar}
    Let $\nu\in\Ctwoa_{loc}(NL)$ be bounded and such that $\L_{L}\nu=0$. Then $\nu\equiv0$. 
\end{proposition}
\begin{proof}
We first show that it is enough to prove that bounded holomorphic sections of the normal bundle of $L\subset \TN$ are identically zero. Recall that a section $\nu$ of the (real) normal bundle is holomorphic if $\bar{\partial}\nu(X):=\nabla^\perp_{JX} \nu-J(\nabla^\perp_X \nu)=0$ for every $X$ tangent to $L$.

For every $0<T$, let $\chi:=\chi_{T,2T}$ be the cutoff function defined in \cref{sec: construction initial surface}. From the Weitzenb\"ock identity of \cite{Simons1968}*{Theorem 3.5.1}, we know that if $\L_L\nu=0$ then 
\[
0=-\int_L \langle \L_L\nu,\chi\nu\rangle=\int_{L} \langle \bar{\partial}^*\bar{\partial}\nu,\chi\nu\rangle.
\]
The section $\chi\nu$ is compactly supported and we can, therefore, freely integrate by parts to obtain the following chain of inequalities:
\begin{equation*}
	\int_{B_{T}} \abs{\bar{\partial}\nu}^2\leq \int_{L} \chi\abs{\bar{\partial}\nu}^2=-\int_L  \langle \bar{\partial}\nu,\nabla \chi\ast\nu\rangle\leq \norm{\nu}_{C^{0}}\left(\int_{B_{2T}\setminus B_{T}}\abs{\bar{\partial}\nu}^2\right)^{\frac{1}{2}}\left(\int_{B_{2T}\setminus B_{T}}\abs{\nabla\chi}^2\right)^{\frac{1}{2}},
\end{equation*}
where $B_{T}:=\pi^{-1}(\EBall{T})\cap L$. We can now use $\abs{\nabla\chi}\leq C T^{-1}$, the linear volume growth of the cigar and Peter-Paul inequality to obtain
\[
\int_{B_{T}} \abs{\bar{\partial}\nu}^2\leq \frac{C^2\norm{\nu}^2_{C^{0}}}{\epsilon^2 T}+\epsilon^2 \left(\int_{B_{2T}}\abs{\bar{\partial}\nu}^2\right).
\]
Choosing $\epsilon>0$ sufficiently small, an iteration scheme (cfr. \cite{Simon1996}*{Chapter 2.8; Lemma 2} which, although stated for balls in $\R^n$, we can apply because $L$ has bounded geometry) implies:
\[
\int_{B_{T}} \abs{\bar{\partial}\nu}^2\leq \frac{C'\norm{\nu}^2_{C^{0}}}{\epsilon^2 T}.
\]
By letting $T\to\infty$ we conclude that bounded solutions of $\L_L\nu=0$ also solve $\bar{\partial}\nu=0$.

We now prove that bounded holomorphic sections of the normal bundle are trivial. First, we construct a global holomorphic trivialization, which corresponds to $\partial_{z_2}$ under the identification of the cigar with $\C\times\{0\}\subset\C^2$. Consider the orthonormal frame on $\TN$ defined by $e_0:=\phi_{\TN}^{1/2}\partial_t$, $e_i=\phi_{\TN}^{-1/2}\tilde{\partial}_{x_i}$ for $i=1,2,3$, where $\partial_t$ is the generator of the $S^1$-action and $\tilde{\partial}_{x_i}$ is the horizontal lift to $\TN$ of $\partial_{x_i}$: $e_0,e_1$ and $e_2,e_3$ span the tangent and, respectively, normal space of $L$. Moreover, the complex structure $J$ acts on this frame by $J e_0=-e_1$, $J e_1=e_0$, $J e_2=e_3$ and $J e_3=-e_2$. Hence, $e_2-ie_3$ is of type $(1,0)$ and all the sections of the normal bundle of type $(1,0)$ are of the form $a(e_2-ie_3)$ for some $a\in C^{\infty}(L,\C)$. Using the first structure equations, we can compute $\nabla^\perp$ in this frame (see for instance \cite{Trinca2022}*{Lemma 2.1}):
\begin{equation*}
    \begin{aligned}
        \nabla^\perp_{e_0} e_2&=\frac{1}{2\phi^{3/2}}\DerParz{\phi}{x_1} e_3, \hspace{8pt} \qquad \qquad \nabla^\perp_{e_1} e_2=0,\\
        \nabla^\perp_{e_0} e_3&=-\frac{1}{2\phi^{3/2}}\DerParz{\phi}{x_1} e_2,   \qquad \qquad \nabla^\perp_{e_1} e_3=0.
    \end{aligned} 
\end{equation*}
From these equations, we deduce that $a(e_2-ie_3)$ is holomorphic if and only if
\begin{equation*}\label{eqn: holomorphic normal}
    e_1(a\phi^{-1/2})=-ie_0(a\phi^{-1/2}).
\end{equation*}
The solution $a=\phi^{1/2}$ yields an holomorphic section that blows up at the origin. One can verify instead that $(\rho e^{-it}\phi^{1/2}) (e_2-ie_3)$ is a nowhere vanishing section of the normal bundle on the whole of $L$, where $\rho=\sqrt{x_1}e^{x_1}$. Since it grows (superexponentially) at infinity, we conclude that there are no bounded holomorphic sections of the normal bundle.
\end{proof}

\begin{remark*}
Constant multiples of the singular holomorphic section $a=\phi^{1/2}$ correspond to the family of holomorphic submanifolds corresponding to $x_2+ix_3=c$, $c\in \C$: the fact that the holomorphic section is singular at the origin corresponds to the fact that the generic fibre of this family is a cylinder $\C^\ast$ while the central fibre (of which the cigar is an irreducible component) is nodal. 
\end{remark*}

We now explain how the cigar in the Taub--NUT space can be regarded as a ``building block'' for the initial surface constructed in the previous section. Let $\Sc^d\subset X_d$ be the surface constructed in \cref{sec: aprroximate minimal surface}. For every fixed $R>0$ and $i=\pm1,\pm2$, we can use \cref{lemma: metric estimates near punctures} to identify $\pi^{-1}(\EBallc{R}{p_i})\subset X_d$ with $\pi_{\TN}^{-1}(\EBall{R})$ in the Taub--NUT space. Given $d$ large enough, each piece of initial surface $\Sc^d\cap \pi^{-1}(\EBallc{R}{p_i})\subset X_d$ can be identified with a piece of the cigar $L\cap \pi_{\TN}^{-1}(\EBall{R})$ in the Taub--NUT space (cfr. \cref{sec: construction initial surface}). We can identify the normal bundle of $\Sc^d\cap \pi^{-1}(\EBallc{R}{p_i})$ with the normal bundle of $L\cap \pi_{\TN}^{-1}(\EBall{R})$. Moreover, from \cref{lemma: metric estimates near punctures} we can see that the linearized operator $\L_{\Sc^d}$ on $\Sc^d\cap \pi^{-1}(\EBallc{R}{p_i})$ smoothly converges, as $d\to\infty$, to the linearized operator $\L_L$ on $L\cap\pi_{\TN}^{-1}(\EBall{R})$.

\subsubsection{Totally geodesic cylinder}\label{sec: linear analysis on cylinder} Let $C$ be the totally geodesic cylinder $\R\times \{0\}\times S^1\subset (\R^3\times S^1, dt^2+g_{\R^3})$,  let $\L_C$ be the Jacobi operator of $C$ and let $\G_C$ be the subgroup of isometries of $\R^3\times S^1$ which is generated by the maps:
\begin{align*}
R_1(x_1,x_2,x_3,e^{it})&=(x_1,-x_2,x_3,e^{-it})\\
R_2(x_1,x_2,x_3,e^{it})&=(x_1,x_2,-x_3,e^{-it})\\
R_3(x_1,x_2,x_3,e^{it})&=(x_1,x_2,x_3,-e^{it}).
\end{align*}
Obviously, $C$ is invariant under $\G_C$.

\begin{proposition}\label{prop: linearized problem on cylinder}
    Let $\nu\in\Ctwoa_{loc}(NC)$ be a $\G_C$--equivariant bounded solution to $\L_{C}\nu=0$. Then $\nu\equiv0$.
\end{proposition}
\begin{proof}
The normal bundle of $C$ is naturally trivialized by $(\partial_{x_2},\partial_{x_3})$ and hence, in this trivialization, the linearized operator for the mean curvature has the form:
\[
\L_{C}=(\Delta_C,\Delta_C),
\]
where $\Delta_C=\partial_{x_1}^2+\partial_t^2$ is the standard Laplace operator on a flat cylinder. It is then clear that the bounded kernel of $\L_{C}$ consists of constant functions. Since $(R_i)_{\ast}\partial_{x_i}=-\partial_{x_i}$ for $i=1,2$, we conclude that such constants need to be zero.
\end{proof}
We now explain how the totally geodesic cylinder can be regarded as a ``building block'' for the initial surface constructed in the previous section. Let $\Sc^d\subset X_d$ be the surface constructed in \cref{sec: aprroximate minimal surface}. For every $0<R<d$, we can use \cref{lemma: exp gauge} to $\G$-equivariantly identify $\pi^{-1}(\EBall{R})$ with $\EBall{R}\times S^1$. Pick now $c_i:=\lambda_d p_i\in \R^3$ for $i=\pm1,\pm2$ such that $\lambda_d\to\infty$ and $\lambda_d-d\to\infty$. For any fixed $k>0$, consider the neighbourhoods $\EBallc{k}{c_i}\times S^1\subset \EBall{R}\times S^1$ around each $c_i$. This makes sense for some $R,d$ big enough. Each piece of initial surface $\Sc^d\cap\pi^{-1}(\EBallc{k}{c_i})\subset X_d$ can be identified with $C\cap (\EBall{k}\times S^1)\subset \R^3\times S^1$, up to small perturbations (cfr. \cref{sec: construction initial surface}). The action of $\G$ is transitive on the four cylinders corresponding to $i=\pm 1, \pm 2$ and the stabiliser of one such cylinder is isomorphic to the subgroup $\G_C$. We can therefore consider only one piece of totally geodesic cylinder in $\EBall{k}\times S^1$ invariant under the action of $\G_C$. We can $\G_C$--equivariantly identify the normal bundle of $\Sc^d\cap \pi^{-1}(\EBallc{k}{c_i})$ with the normal bundle of $C\cap (\EBall{k}\times S^1)$. Moreover, from \cref{lemma: metric estimates near 0} we can deduce that the linearized operator $\L_{\Sc^d}$ on $\Sc^d\cap \pi^{-1}(\EBallc{k}{c_i})$ smoothly converges, as $d\to\infty$, to the linearized operator $\L_C$ on the piece of cylinder $C\cap\EBall{k}\times S^1$.

\subsection{The linear problem}
In this subsection, we prove the main result about the Jacobi operator $\L$ of the initial surface $\Sc^d$, i.e., that it has uniformly bounded inverse when restricted to $\G$--equivariant sections. We begin with the following preliminary result.

\begin{proposition}[Uniform Schauder estimates]\label{prop: uniform elliptic estimates}
    Let $\alpha\in(0,1)$. For every $d$ big enough, there exists a constant $C>0$ independent from $d$ such that:
    \[
    \norm{\nu}_{\Ctwoa}\leq C \left(\norm{\L\nu}_{\Czeroa}+\norm{\nu}_{\Czero}\right)
    \]
    for every $\nu\in\Ctwoa(N\Sc^d)$.
\end{proposition}
\begin{proof}
The estimate is obtained by taking supremums of the analogous local interior estimates and therefore it suffices to argue that $\Sc^d$ has bounded geometry as a submanifold of $X_d$ (i.e., $\Sc^d$ has bounded geometry as a Riemannian manifold and its second fundamental form and its covariant derivatives are uniformly bounded on balls of definite size). This is clear since, as explained, $\Sc^d$ is made up of pieces which for $d$ large enough can be regarded as small $\C^\infty$--perturbations of the Scherk surface and totally geodesic cylinder in $\R^3\times S^1$ and the cigar in the Taub--NUT space.
\end{proof}

We can now prove the main results of this subsection.
\begin{proposition}[Uniform estimate for the inverse]\label{prop: uniform linear estimates}
    For every $d$ large enough, there exists a constant $C>0$ independent from $d$ such that
    \[
    \norm{\nu}_{\Ctwoa}\leq C\norm{\L\nu}_{\Czeroa}
    \]
    for every $\nu\in\Ctwoa_\G(N\Sc^d)$.
\end{proposition}
\begin{proof}
    Assume by contradiction this is not the case. Then there exists a sequence $d_k\to\infty$ and $\nu_k\in\Ctwoa_\G(N\Sc^{d_k})$ such that $\norm{\nu_k}_{\Ctwoa}=1$ and $\norm{\L\nu_k}_{\Czeroa}\to0$. Applying \cref{prop: uniform elliptic estimates} to such a sequence, we deduce:
    \begin{align*}
        \norm{\nu_k}_{\Czero}\geq C_0>0,
    \end{align*}
    where $C_0$ is some positive constant independent from $d$. It follows that, for every $k$, there exists $x_k\in \Sc^{d_k}$ such that $\abs{\nu_k(x_k)}\geq C_0>0$. Up to subsequences, at least one of the following three situations holds:
    \begin{enumerate}
        \item\label{case: Scherk} $\dist_{k}(x_k,p)$ is bounded in $k$;
        \item\label{case: cigar} $\dist_k(x_k,p_i)$ is bounded in $k$ for some $i=\pm1,\pm2$;
        \item\label{case: cilinder} $\dist_k(x_k,p)\to+\infty$ and $\dist_k(x_k,p_i)\to+\infty$ as $k\to\infty$ for every $i=\pm1,\pm2$, 
    \end{enumerate}
    where $p$ is the chosen base point in $\pi^{-1}(0)$ and $\dist_k$ is the distance function associated to the Riemannian metric $g_{d_k}$. We only discuss \cref{case: Scherk}, as \cref{case: cigar,case: cilinder} are analogous by considering \cref{sec: linear analysis on cigar} and \cref{sec: linear analysis on cylinder} instead of \cref{sec: linear analysis on Scherk}, respectively. 
    
    For every fixed $R>0$, we restrict $\nu_k$ to $\Sc^d\cap \pi^{-1}(\EBall{R})$ and we denote this restriction by $\nu^R_k$. Since $\norm{\nu_k}_{\Ctwoa}=1$ for every $k$, we must also have $\norm{\nu^R_k}_{\Ctwoa}\leq1$. By Arzel\`a--Ascoli theorem, $\nu^R_k$ converges in $\Ctwob$, up to subsequences, to some $\nu^R_\infty$ for every $0<\beta<\alpha<1$ (cfr. discussion after \cref{prop: linearized problem on Scherk}). Using a standard diagonal argument for $k$ and a sequence $R_i\to\infty$, we can find $\nu_\infty$ such that $\nu_k\to\nu_\infty$ in $\Ctwob_{\Sc,\loc}$. Since $\nu_k$ is a $\G$-equivariant element of $\ker\L$ for every $k$ and $\L \nu_k\to \L_{\Sc}\nu_\infty$ in $\Czerob_{\Sc,\loc}$, we have that $\nu_\infty$ is a $\G$--equivariant element of $\ker\L_{\Sc}=0$. Finally, up to passing to a further subsequence, $x_k$ as in \cref{case: Scherk} must converge to some $x_\infty\in \Sc$ and $|\nu_\infty(x_\infty)|>0$. 

    We now show that $\nu_\infty$ is a bounded section of $N\Sc$. First, note that $\norm{\nu_k}_{\Ctwoa}=1$ implies that $\norm{\nu_k}_{\Czero}\leq 1$; hence for every point $x\in\mathcal{S}$ we can find a $k$ big enough such that $\abs{\nu_k(y)}\leq1$ for every $y$ in a small ball centred at $x$. We conclude that $\abs{\nu_\infty (x)}\leq 1$ because $\nu_k\to \nu_\infty$ in $\Ctwob_{\Sc,\loc}$.
    

    Since we have constructed a $\G$--equivariant bounded section $\nu_\infty$ of $\Ctwob_{\Sc,\loc}$ which is in $\ker\L_\Sc$, we can use \cref{prop: linearized problem on Scherk} to deduce that it must be identically zero. This however is a contradiction to $|\nu_\infty(x_\infty)|>0$.
\end{proof}

\subsection{The non-linear problem} We are finally ready to deform the approximate minimal surface $\Sc^d\subset X_d$ into an actual minimal surface using the following quantitative version of the Implicit Function Theorem.
\begin{lemma}\label{lemma: contraction map}
Let $E$ and $F$ be Banach spaces and let $\Phi: E\to F$ be a smooth function between them. Writing $\Phi(x)=\Phi(0)+\L(x)+N(x)$ with $\L$ linear map of $\Phi$ at $0$, if there exists constants $r,C,q$ such that:
\begin{enumerate}
    \item $\L$ is invertible with $\norm{\L^{-1}}\leq C$,
    \item $\norm{N(x)-N(y)}_{F}\leq q\norm{x+y}_E\norm{x-y}_E$ for all $x,y\in B_r(0)\subset E$;
    \item $\norm{\Phi(0)}_F<\min\{\frac{r}{2C},\frac{1}{4C^2q}\}$;
\end{enumerate}
then there exists a unique $x\in \overline{B}_{2C\norm{\Phi(0)}}(0)\subset E$ such that $\Phi(x)=0$.
\end{lemma}
\begin{proof}
Consider the map $G:\overline{B}_{2C\norm{\Phi(0)}}\subset E\to \overline{B}_{2C\norm{\Phi(0)}}\subset E$ defined as:
\[
Gx:=-\L^{-1}\left(\Phi(0)+N(x)\right).
\]
We have $Gx=x$ if and only if $\Phi(x)=0$ and the quantitative estimates in the assumptions guarantee that $G$ is a contraction.
\end{proof}

The final ingredient to apply \cref{lemma: contraction map} in our setting is a uniform estimate on the non-linear terms. 
\begin{proposition}[Uniform estimates for non-linearities]\label{prop: uniform nonlinearities estimates}
    Let $\alpha\in(0,1)$. For every $d$ large enough, there exists a constant $q>0$ and $r>0$ independent from $d$ such that
    \[
    \norm{Q\nu-Q\nu'}_{\Czeroa}\leq q \left(\norm{\nu-\nu'}_{\Ctwoa}\cdot \norm{\nu+\nu'}_{\Ctwoa}\right)
    \]
    for every $\nu,\nu'\in B_r(0)\subset \Ctwoa(N\Sc^d)$ .
\end{proposition}
\begin{proof}
    The mean curvature operator $H$ is smooth \cite{White1991}*{Theorem 1.1} and, via a cut and paste argument similar to the one in \cref{prop: uniform elliptic estimates}, it has uniformly controlled norm in a uniform neighbourhood of the origin. The proposition now follows from the Taylor expansion theorem at the origin.
\end{proof}

\begin{theorem}\label{thm: Main theorem}
Let $(X_d,g_d)$ be the multi-taub-NUT space with $m=1$ and characterizing points as in \cref{eqn: characterizing points Xd}. There exists $d_0>0$ such that, for every $d>d_0$ and $p\in\pi^{-1}(0)\cong S^1$, there is a distinct minimal surface $\Sigma\subset X_d$ with the following properties:
\begin{enumerate}
\item\label{item: S diff S2} the surface $\Sigma$ is diffeomorphic to $S^2$;
\item\label{item: S equiv} the surface $\Sigma$ is $\G$--equivariant, where $\G$ is as defined in \cref{sec: symmetries Xd} choosing the characterizing point for $\G$ to be $p$;
\item\label{item: S Gauss map} the degree of the (positive) Gauss lift $a:\Sigma\to S^2$ is $1$;
\item\label{item: Null(S)} the nullity of $\Sigma$ satisfies $\Null(\Sigma)\geq1$;
\item\label{item: Ind(S)} the (Morse) index of $\Sigma$ satisfies $\Ind(\Sigma)\geq1$;
\item\label{item: S homology} if $H_2(X_d;\Z)$ (isomorphic to the $A_3$ lattice with its intersection form) is generated by the circle-invariant minimal spheres: $c_1:=[\pi^{-1}([p_{-1},p_{-2}])]$, $c_2:=[\pi^{-1}([p_{-2},p_{1}])]$ and $c_3:=[\pi^{-1}([p_{1},p_{2}])]$, then $[\Sigma]=c_1+c_3$. 
\end{enumerate}
\end{theorem}
\begin{proof}
Let $\Sc^d\subset X_d$ be the surface constructed in \cref{sec: aprroximate minimal surface} for any fixed $p\in\pi^{-1}(0)$. Let $E:=\Ctwoa_{\G}(N\Sc^d)$ and let $F:=\Czeroa_\G (N\Sc^d)$ for some $\alpha\in(0,1)$. Finally, let $\Phi:E\to F$ be the mean curvature operator. For every $d$ large enough, we can apply \cref{lemma: contraction map}. Indeed, the first two conditions are true for constants uniform in $d$ by \cref{prop: uniform linear estimates} (which proves injectivity of the linearisation restricted to $\G$--equivariant sections, and hence also surjectivity since this restriction is a self-adjoint operator) and \cref{prop: uniform nonlinearities estimates}, while the mean curvature of $\Sc^d$ (i.e., $\Phi(0)$ in the notation of the lemma) can be made arbitrary small using \cref{prop: mean curvature estimates initial surface} as $d\to\infty$. In this way, we have constructed $\nu\in \Ctwoa_{\G}(N\Sc^d)$ such that $\norm{\nu}_{\Ctwoa}\leq C\log d/d^2$ and such that $\Sigma:=\graph(\nu)$ is minimal. By standard elliptic regularity, $\Sigma$ is a smooth minimal surface and, by construction (cfr. \cref{remark: properties ScIn}), $\Sigma$ satisfies \cref{item: S diff S2,item: S equiv,item: S Gauss map}. \cref{item: Null(S)} is also clear as our construction produces an $S^1$--family of minimal surfaces corresponding to the choice of base point $p\in \pi^{-1}(0)$ or equivalently the isometric $S^1$-action on $X_d$.


We now construct vector fields $V_d\in \Gamma(N\Sc^d)$ such that $-\int_{\Sc^d}\langle \L_{\Sc^d} V_d,V_d\rangle<-c<0$ for some constant $c$ independent from $d$. This is enough to prove \cref{item: Ind(S)} as $\Sigma$ is arbitrarily close in $C^{2,\alpha}$ to $\Sc^d$ for $d\to\infty$. We construct $V_d$ as a piecewise smooth normal vector field supported in the region of $\Sc^d$ that coincides with a portion of the Scherk surface. Let $T=W_0(d^2)$ be as in \cref{sec: aprroximate minimal surface}, $\Omega_T:=\Sc\cap (\EBall{T}\times S^1)\subset \R^3\times S^1$ and let $\tilde{V}_d$ be the first eigenvector field with Dirichlet boundary condition on $\Omega_T$, which we trivially extend on the whole $\Sc$. Such a vector field always exists by standard PDE techniques; if $d$ is sufficiently large it also satisfies $-\int_{\Sc}\langle \L_{\Sc} \tilde{V}_d,\tilde{V}_d\rangle<0$ because the Scherk surface has Morse index 1, cfr. \cite{MontielRos1991}.
Note that the variational characterization of eigenvalues implies that $-\int_{\Sc}\langle \L_{\Sc} \tilde{V}_d,\tilde{V}_d\rangle$ is a decreasing quantity in $d$ and, via \cref{lemma: exp gauge}, we can construct $V_d\in\Gamma(N\Sc^d)$. The strict inequality of the second variation along $V_d$ continues to hold because of \cref{lemma: metric estimates near 0}.


Finally, we prove \cref{item: S homology} using a Mayer--Vietoris argument. Consider the covering of $X_d$ given by $U:=\pi^{-1}(\{(x_1,x_2,x_3)\in\R^3: x_1 x_2<2\}\cap \{(x_1,x_2,x_3)\in\R^3: x_1 x_2>-2\})$ and $V:=\pi^{-1}(\{(x_1,x_2,x_3)\in\R^3: x_1 x_2>1\}\cup \{(x_1,x_2,x_3)\in\R^3: x_1 x_2<-1\}$. Observe that $U\cong \R^3\times S^1$ does not contain any $p_i$ for $d$ big, while $V$ consists of four connected components $V_i\cong\R^4$ each containing exactly $p_i$ for $i=\pm1,\pm2$. The intersections $U\cap V_i\cong \R^3\times S^1$ are disjoint. The Mayer--Vietoris sequence reads
\[
0\cong H_2(U)\oplus H_2(V)\to H_2(X)\xrightarrow[]{i} H_1(U\cap V)\cong\Z^4\xrightarrow[]{p} H_1(U)\oplus H_1(V)\cong\Z\oplus\{0\}\to H_1(X)\cong \{0\},
\]
from which we deduce that $H_2(X)\cong\im(i)=\ker(p)$. Under the obvious identification $H_1(U\cap V)\cong \Z^4$ and $H_1(U)\cong\Z$ consistent with the orientations, the map $p:\Z^4\to\Z$ takes the form: $(n_1,n_2,n_3,n_4)\mapsto n_1+n_2+n_3+n_4$. From \cref{sec: construction initial surface}, we see that $i_\ast[\Sigma]=i_\ast [\Sc^d]=(1,-1,1,-1)=(1,-1,0,0)+(0,0,1,-1)$. We can conclude because $i_\ast c_1=(1,-1,0,0)$, $i_\ast c_2=(0,1,-1,0)$ and $i_\ast c_3=(0,0,1,-1)$ under this identification. 
\end{proof}


\begin{remark}
The construction works more generally under small $\G$-equivariant perturbations of the Riemannian metric $g_d$. For instance, our argument works for multi-Taub--NUT gravitational instantons with characterizing points $\{p_{\pm1},p_{\pm2},q_1,...,q_n\}$ for $\{q_i\}_{i=1}^n$ distributed in a $\G$-invariant way and with $\abs{q_i}$ big enough for every $i=1,...,n$.
\end{remark}

\section{Generalizations and applications}\label{sec:Applications}
In this section, we explain how to generalize our construction in several directions.
Firstly, we can use the first-named author's construction of hyperk\"ahler metrics on the K3 manifold \cite{Foscolo2019} to exhibit the existence of unstable minimal spheres with degree-1 positive Gauss lift for certain hyperk\"ahler metrics on the K3 manifold. Secondly, by taking several periods of Scherk's surface, we observe that there are unstable minimal surfaces of any topological type in $X_d$ for $d$ large enough. Similarly, we can use Karcher's saddle towers to construct unstable minimal surfaces of any topological type in ALF multi-Taub--NUT spaces with characterizing points lying on the vertices of a $2n$-sided regular polygon with sufficiently large side.
Finally, we show that for any minimal sphere with degree-1 Gauss lift we produce, there exists a stationary harmonic Fueter map with a $0$-dimensional singular set and a stationary harmonic tri-holomorphic map with $1$-dimensional singular set.

\subsection{Minimal spheres in K3 surfaces} We first briefly recall the first-named author's construction of hyperk\"ahler metrics on the K3 manifold (i.e., the smooth 4-manifold underlying any complex K3 surface) in a highly symmetric setting. Afterwards, we explain how to construct unstable minimal surfaces of degree-1 positive Gauss-lift with respect to these hyperk\"ahler metrics on the K3 manifold near the collapsed limit. 

\subsubsection{The highly-symmetric flat torus}\label{sec: highly symmetric torus} Let $(\T^2\times S^1,g_{\T^2}+g_{S^1})$ be a 3-torus which can be decomposed into a squared 2-torus $(\T^2,g_{\T^2})$ of side $\alpha$ and a circle $(S^1,g_{S^1})$ of length $\beta$. Let $ \tilde \tau:\T^2\times S^1\to\T^2\times S^1$ be the standard involution on $\T^2\times S^1$ with fixed points $q_1,...,q_8$ which, in the lattice $\R^2/\alpha\Z^2\times \R/\beta\Z$, are 
	\begin{equation}
		\begin{aligned}\label{eqn: q points tori}
			&q_1:=[(0,0,0)], \quad\quad q_2:=[(\alpha/2,0,0)], \quad\quad q_3:=[(0,\alpha/2,0)], \quad\quad\hspace{3pt} q_4:=[(\alpha/2,\alpha/2,0)],\\
			&q_5:=[(0,0,\beta/2)], \quad q_6:=[(\alpha/2,0,\beta/2)], \quad q_7:=[(0,\alpha/2,\beta/2)], \quad q_8:=[(\alpha/2,\alpha/2,\beta/2)].
		\end{aligned}
	\end{equation}
We now pick the following 8 points of $\T^2\times S^1$: 
\begin{equation}
		\begin{aligned}\label{eqn: p points tori}
		&{p}_1:=[(\alpha/4,\alpha/4,\beta/4)], \quad \quad {p}_2:=[(3\alpha/4,\alpha/4,\beta/4)], \quad {p}_3:=[(\alpha/4,3\alpha/4,\beta/4)],\\
		 &{p}_4:=[(3\alpha/4,3\alpha/4,\beta/4)], \quad {p}_5:=[(\alpha/4,\alpha/4,3\beta/4)], \quad {p}_6:=[(3\alpha/4,\alpha/4,3\beta/4)],\\
		 & {p}_7:=[(\alpha/4,3\alpha/4,3\beta/4)], \quad {p}_8:=[(3\alpha/4,3\alpha/4,3\beta/4)],
		\end{aligned}
\end{equation}
which satisfy $\tilde \tau (p_i)=p_{9-i}$ for every $i=1,...,8$. 

For this choice of points, the torus $\T^2\times S^1$ admits a large discrete subgroup $\tilde\G$ of isometries containing $\tilde \tau$ and preserving the sets $\{q_1,...,q_8\}$ and $\{p_1,...,p_8\}$. The group $\tilde\G$ is generated by the involutions
\begin{align*}
         \Tilde{R}_1[(x_1,x_2,x_3)]&=[(\tfrac{\alpha}{2}-x_1,x_2,x_3)], & \Tilde{R}_2[(x_1,x_2,x_3)]&=[(x_1,\tfrac{\alpha}{2}-x_2,x_3)], \\
          \Tilde{R}_3[(x_1,x_2,x_3)]&=[(x_1,x_2,\tfrac{\beta}{2}-x_3)], & \Tilde{T}[(x_1,x_2,x_3)]&=[(x_1,x_2,\tfrac{\beta}{2}+x_3)],\\
         \Tilde{R}_4[(x_1,x_2,x_3)]&=[(x_2,x_1,x_3)], & \Tilde{R}_5[(x_1,x_2,x_3)]&=[(\tfrac{\alpha}{2}-x_2,\tfrac{\alpha}{2}-x_1,x_3)],\\
         \Tilde{R}'_4[(x_1,x_2,x_3)]&=[(\tfrac{\alpha}{2}+x_2,\tfrac{\alpha}{2}+x_1,x_3)], & \Tilde{R}'_5[(x_1,x_2,x_3)]&=[(-x_2,-x_1,x_3)].
\end{align*}
Note that $\tilde{\tau}=\Tilde{T}\circ \Tilde{R}_3\circ \Tilde{R}_4\circ\Tilde{R}'_5$. The orbit of the point $p_1$ consists of the points $\{ p_1, p_4, p_5, p_8\}$ and $p_1$ has stabiliser the subgroup $\tilde\G_{p_1}$ generated by $\Tilde{R}_1, \Tilde{R}_3$ and $\Tilde{R}_4$, isomorphic to $\G\simeq D_4\times \Z_2$. Similarly, the stabiliser of the point $p_2$ is a second copy of $\G$, the subgroup $\tilde\G_{p_2}$ generated by $\Tilde{R}_1, \Tilde{R}_3$ and $\Tilde{R}'_4$, and $\tilde\G\cdot p_2 = \{ p_2, p_3, p_6, p_7\}$. Finally, all the points $q_1, \dots, q_8$ belong to the same $\tilde\G$--orbit and the stabiliser of $q_1$ is the subgroup $\tilde\G_{q_1}$, isomorphic to $\Z_2\times \Z_2\times \Z_2$, generated by $\Tilde{T}\circ\Tilde{R}_3, \Tilde{R}_4$ and $\Tilde{R}'_5$. All these facts can be checked by explicit calculations. 


\subsubsection{The construction of the invariant hyperk\"ahler metrics} Now, by taking each of the punctures $p_i$'s with weight 4 we satisfy the balancing condition \cite{Foscolo2019}*{Equation 4.1}. Then on the punctured torus $\T^\ast= (\T^2\times S^1)\setminus\{ p_1, \dots, p_8,q_1,\dots, q_8\}$ there exists an abelian monopole $(\phi,\theta)$, i.e. a solution of \cref{eqn: monopole equation}, with prescribed singular behaviour at the punctures and which determines a principal $\U(1)$--bundle $P\to \T^\ast$ with connection $\theta$. Moreover, the involution $\tilde \tau$ lifts to an involution $\tau$ on $P$ (uniquely up to gauge transformations), which acts as the standard involution on the $\U(1)$ fibres (cfr. \cite{Foscolo2019}*{Proposition 4.3}). For every $\epsilon$ small enough, the Gibbons--Hawking ansatz with respect to $\phi_\epsilon=1+\epsilon \phi$ gives an incomplete $\tau$-invariant hyperk\"ahler metric $g^{gh}_\epsilon$ on $\restr{P}{U_\epsilon}$, where $U_{\epsilon}$ is the complement of (arbitrarily) small balls centred at the punctures (cfr. \cite{Foscolo2019}*{Lemma 4.9}). The gluing parameter $\epsilon$ determines the size of the $\U(1)$ fibre (cfr. \cite{Foscolo2019}*{Eq. 4.6}). After taking the $\tau$-quotient, the singular behaviour of $(\phi,\theta)$ implies that we can ``fill'' each removed neighbourhood of $q_i$ with a compact portion of an appropriately rescaled copy of the $D_0$ ALF gravitational instanton, the (not simply connected) Atiyah--Hitchin manifold, and the removed neighbourhoods of $p_i$ with a rescaled copy of a compact portion of $X_d$ (a $D_4\times \Z_2$--invariant $A_3$ ALF gravitational instanton) for some fixed $d$ larger that the costant $d_0$ given by \cref{thm: Main theorem}. In the transition region, we suitably interpolate between the aforementioned hyperk\"ahler metrics (cfr. \cite{Foscolo2019}*{Equation 5.7}).

In this way, we have described an approximate hyperk\"ahler manifold $(M_\epsilon,\tilde g_\epsilon)$ which, for every $\epsilon$ small enough, can be perturbed \cite{Foscolo2019}*{Theorem 6.15} (in a unique way using the slice to hyperk\"ahler deformations and diffeomorphisms of this theorem) to genuine hyperk\"ahler manifold $(M_\epsilon,  g_\epsilon)$. As $\epsilon \to 0$, $(M_\epsilon,  g_\epsilon)$ collapses to the flat orbifold $\T^\ast/\Z_2$ with bounded curvature away from the punctures, while after rescaling to fixed $\epsilon$ the manifold smoothly converges to the $D_0$ ALF space on compact subsets containing one of the $q_j$'s and to $X_d$ on compact subsets containing one of the $p_i$'s.

We will now show that we can lift the action of $\tilde\G$ to $M_\epsilon$ as a group of isometries of $\tilde{g}_\epsilon$ so that the action of the stabiliser of $p_i$ on the piece of $X_d$ used in the gluing construction coincides with the given action of $\G$. Then the construction of \cite{Foscolo2019}*{Theorem 6.15} can be carried in a $\tilde\G$--equivariant way so the resulting hyperk\"ahler manifold $(M_\epsilon,g_\epsilon)$ also satisfies ${\tilde\G}<\Iso(M_\epsilon,{g}_\epsilon)$. 

First of all, observe that we can cover the torus $\T^2\times S^1$ $\tilde\G$--equivariantly with open sets each isometric to either a neighbourhood of the origin in $\R^3$ with the action of $\tilde\G_{p_1}\simeq\G\simeq D_4\times \Z_2$ given in \cref{sec: symmetries Xd}, or a neighbourhood of the origin in $\R^3$ with the action of $\tilde\G_{q_1}\simeq\Z_2\times\Z_2\times\Z_2$ generated by \[
a(x_1,x_2,x_3)=(x_2,x_1,x_3), \qquad b(x_1,x_2,x_3)=(-x_2,-x_1,x_3), \qquad c(x_1,x_2,x_3) =(x_1,x_2,-x_3).
\]
(Here $a,b,c$ are simply the lifts to $\R^3$, thought of as the tangent space to $q_1$ of $\T^2\times S^1$, of $\Tilde{R}_4, \Tilde{R}'_5$ and, respectively, $\Tilde{T}\circ\Tilde{R}_3$.) Subsets of the first type are the interior of rectangular parallelepipeds centred at one of the $p_i$'s which are translates of the subset
\[
\{ (x_1,x_2,x_3)\, |\, |x_1+x_2| < \tfrac{\alpha}{4}, |x_1-x_2|<\tfrac{\alpha}{4}, |x_3|<\tfrac{\beta}{4}\}
\]
of $\R^3$ centered at the origin. Subsets of the second type are rectangular parallelepipeds centred at one of the $q_j$'s  which are translates of the subset
\[
\{ (x_1,x_2,x_3)\, |\, |x_1| < \tfrac{\alpha}{4}, |x_2|<\tfrac{\alpha}{4}, |x_3|< \tfrac{\beta}{4}\}
\]
of $\R^3$ centered at the origin. We then have a corresponding decomposition of $M_\epsilon$ by open sets that are diffeomorphic to either an open set of $X_d$ or an open set of the Atiyah--Hitchin manifold, with natural isometric actions of $\G$ and, respectively, $\Tilde\G_{q_1}/\langle abc =\tilde{\tau}\rangle \simeq \Z_2\times\Z_2$ on the base of the circle fibrations of their asymptotic models at infinity. In order to show that $\tilde\G$ acts by isometries on $(M_\epsilon,\tilde g_\epsilon )$ it is therefore enough to argue that these asymptotic actions lift to the asymptotic models and extend to $X_d$ and the Atiyah--Hitchin manifold.

The first case has already been treated in \cref{sec: symmetries Xd}. Consider then the Atiyah--Hitchin manifold. Recall that it admits an isometric cohomogeneity one action of $SO(3)$ with principal orbits $SO(3)/\Z_2\times\Z_2\simeq SU(2)/Q$, where $\Z_2\times \Z_2$ denotes the subgroup of diagonal matrices in $SO(3)$ and $Q$ is the quaternion group. Moreover, the asymptotic circle fibration is induced by $SU(2)/Q\rightarrow SU(2)/\textup{N}$, where $\textup{N}$ is the normaliser of diagonal matrices. (Note that the circle fibres of this fibration must have the opposite orientation with respect to the natural one because of the negative charge $-4$ of the asymptotic Gibbons--Hawking form of the Atiyah--Hitchin metric.) In terms of these group actions the asymptotic region of the Atiyah--Hitchin manifold can therefore be realised as an exterior domain in $\mathbb{H}/Q$ with projection to an exterior domain of $\R^3/\langle \tilde\tau\rangle$ induced by the Hopf projection $q\mapsto qi\overline{q}$. The action of $SO(3)$ on the Atiyah--Hitchin manifold is a lift of the action of $SO(3)$ on $\R^3/\langle \tilde\tau\rangle$ by rotations.
Since $\Tilde\G_{q_1}$ is a subgroup of $O(3)$ and $\tilde\tau$ is orientation reversing, we deduce that the induced action of $\Tilde\G_{q_1}/\langle abc =\tilde{\tau}\rangle\subset SO(3)$ does indeed lifts and extends to an action on the Atiyah--Hitchin manifold.

\subsubsection{The construction of unstable minimal spheres} In the symmetric setting that we have described, we can now use the surface we constructed in \cref{thm: Main theorem} to obtain unstable minimal spheres in K3 surfaces with degree one (positive) Gauss lift. 

\begin{theorem}\label{thm: spheres in K3}
	Let $(M_\epsilon,g_\epsilon)$ be the family of highly symmetric hyperk\"ahler metrics on the K3 surface described above. Then there exists $\epsilon_0>0$ such that, for every $\epsilon<\epsilon_0$, there are unstable minimal spheres of $(M_\epsilon,g_\epsilon)$ with degree one (positive) Gauss lift.
\end{theorem}
\begin{proof}
Consider the four gluing regions diffeomorphic to a subset of $X^d$, and $\tilde \Sigma\subset M_\epsilon$ be the obvious surface obtained from the minimal sphere constructed in \cref{thm: Main theorem}. This is a minimal surface of $(M_\epsilon,\tilde g_\epsilon)$ diffeomorphic to four copies of $S^2$. Since the $\Tilde\G$--equivariant kernel of the linearized operator of $\tilde \Sigma\subset (M_\epsilon,\tilde g_\epsilon)$ is trivial (\cref{prop: uniform linear estimates}), we can use the equivariant version of White's implicit function theorem for minimal immersions \cite{White2017}*{Theorem 2.3}, to perturb $\tilde \Sigma$ to a minimal surface $\Sigma\subset (M_\epsilon,g_\epsilon)$. Each of the four connected components of $\Sigma$ have the desired properties. Indeed, they are obviously spheres with degree one Gauss lift and, up to choosing $\epsilon$ smaller if necessary, they are unstable because of the continuity of the spectrum of the linearized operator. 
\end{proof}

\begin{remark}
	Alternatively, we can more directly mimic the construction of $\Sigma\subset X_d$, and obtain unstable minimal surfaces with degree one (positive) Gauss lift in the highly symmetric tori described above with punctures $q_1,...,q_8$ as in \cref{eqn: q points tori} and $p_1^1,...,p_1^4,...,p_8^1,...,p_8^4$ obtained from the points prescribed in \cref{eqn: p points tori} as follows:
\[
	p_i^j= p_i+r_j, \hspace{15pt} i=1,...,8;\hspace{5pt} j=1,...,4,
\]
where $r_1=(d,d,0)$, $r_2=(-d,d,0)$, $r_3=(d,-d,0)$ and $r_4=(-d,-d,0)$ for some $0<d<\alpha/4$ small enough. At each $q_i$, we glue in a $D_0$ ALF gravitational instanton, while at each $p_i^j$, we glue in a Taub--NUT space. This generalisation yields disjoint unions of unstable minimal 2-spheres in families of $\tilde\G$--invariant hyperk\"ahler metrics on the K3 manifold depending on the four parameters $\epsilon, \alpha, \beta, d$.
\end{remark}

\subsection{Minimal surfaces from multiple Scherk periods}
A straightforward generalization of \cref{thm: Main theorem} can be obtained taking multiple periods of the Scherk surface $\Sc$ instead of one. This means that, for every $g\geq0$, we consider Scherk as a minimal surface of genus $g$ inside of $\R^2\times \R/n\Z\cong \R^2\times S^1$ with $n=g+1$. 

Let $\Sc_g^d\subset X_d$ be the approximate minimal surface of genus $g=n-1$ constructed as in \cref{sec: construction initial surface} from $n$ periods of the Scherk surface $\Sc_g\subset\R^3\times \R/n\Z\cong \R^3\times S^1$. Apart from the degree of the Gauss map and, consequently, of the (positive) Gauss lift, which can be computed using Meeks--Rosenberg formula \cite{MeeksRosenberg1993}*{Theorem 4}, all the other properties of $\Sc_g\subset \R^3\times S^1$ and of the surface $\Sc_g^d\subset X_d$ remain true (cfr. \cref{prop: properties of Scherk}). The perturbation argument for the initial surface $\Sc_g^d\subset X_d$ can be recovered exactly as in \cref{sec: perturbation to minimal surface}. Note that \cite{MontielRos1991}*{Corollary 15} guarantees that the equivariant bounded kernel of the linearized operator of $\Sc_g\subset \R^3\times S^1$ is trivial exactly as in \cref{prop: linearized problem on Scherk}, and that its Morse index is $2g+1$. 

In particular, we can prove the following theorem exactly in the same way as \cref{thm: Main theorem}.
\begin{theorem}\label{thm: multiple periods}
Let $(X_d,g_d)$ be the multi-taub-NUT space with $m=1$ and characterizing points as in \cref{eqn: characterizing points Xd}. For every $g\geq0$, there exists $d_0>0$ such that, for every $d>d_0$ and $p\in\pi^{-1}(0)\cong S^1$, there is a minimal surface $\Sigma^g\subset X_d$ with the following properties:
\begin{enumerate}
\item the surface $\Sigma^g$ is a Riemann surface of genus $g$;
\item the surface $\Sigma^g$ is $\G$-equivariant, where $\G$ is as defined in \cref{sec: symmetries Xd} choosing the characterizing point for $\G$ to be $p$;
\item the degree of the (positive) Gauss lift $a:\Sigma^g\to S^2$ is $(g+1)$;
\item the nullity of $\Sigma^g$ satisfies $\Null(\Sigma^g)\geq1$;
\item the (Morse) index of $\Sigma^g$ satisfies $\Ind(\Sigma^g)\geq 2g+1$.
\end{enumerate}
	
\end{theorem}
 \begin{remark}
 	The surface $\Sigma$ constructed in \cref{thm: Main theorem} coincides with the suface $\Sigma_0$ in this theorem.
 \end{remark}
 
\subsection{Minimal surfaces from saddle towers} Another easy generalization of \cref{thm: Main theorem} can be obtained considering the so-called (Karcher) saddle towers $\mathcal T_n$ for some $n\geq 2$ (and of parameter $\pi/2n$) \cite{Karcher1988} instead of the Scherk surface $\mathcal S$. An explicit description using Weierstra{\ss} representation can be found in \cite{Karcher1988}*{Equation 2.3.1} and in \cite{Weber2005}*{Section 5.5}. 

These surfaces are properly embedded singly periodic minimal surfaces in $\R^3$ such that, when quotiented by one period, i.e., when considered as $\mathcal T_n \subset \R^2\times\R/\Z\cong \R^2\times S^1$, they have genus zero and $2n$ exponentially asymptotically cylindrical ends. The gauss map has degree $n-1$ because of Meeks--Rosenberg formula \cite{MeeksRosenberg1993}*{Theorem 4} and the symmetry group is $\D_{2n}$. Note that $\mathcal T_2\cong\Sc$.

Let $(X^d_n, g_d)$ be the space constructed via the multi-Taub--NUT space with $m=1$ and points $p_1,...,p_{2n}$ lying on the vertices of a $2n$-sided regular polygon centered at the origin and such that $\abs{p_i}=d$ for every $i=1,...,2n$. Analogously to \cref{sec: symmetries Xd}, for every $p\in\pi^{-1}(0)$ there exists a subgroup $\G_{2n}\cong \D_{2n}\times\Z_2$ of the isometry group $\Iso(X^d_n,g_d)$. Moreover, one can construct a local equivariant trivialization near $\pi^{-1}(0)$ (cfr. \cref{lemma: exp gauge}) and prove estimates similar to \cref{lemma: metric estimates near 0} and \cref{lemma: metric estimates near punctures}.

Let $\mathcal T^d_n\subset X^d_n$ be the approximate minimal sphere constructed as in \cref{sec: construction initial surface} with $\mathcal T_n\subset \R^3\times S^1$ instead of $\Sc\subset \R^3\times S^1$. Using the same estimates and perturbation argument, we can prove the following. Note that \cite{MontielRos1991}*{Corollary 15} and \cite{MeeksRosenberg1993}*{Theorem 4} guarantee that the equivariant bounded kernel of the linearized operator of $\mathcal T_n\subset \R^3\times S^1$ is trivial, and that its Morse index is $2n-3$.

\begin{theorem}\label{thm: saddle tower}
Let $n\geq2$ and $(X^d_n,g_d)$ be the multi-taub-NUT space with $m=1$ and characterizing points $p_1,...,p_{2n}$ lying on the vertices of a $2n$-sided regular polygon centered at the origin and such that $\abs{p_i}=d$ for every $i=1,...,2n$. There exists $d_0>0$ such that, for every $d>d_0$ and $p\in\pi^{-1}(0)\cong S^1$, there is a minimal surface $T_n\subset X^d_n$ with the following properties:
\begin{enumerate}
\item the surface $T_n$ is diffeomorphic to $S^2$;
\item the surface $T_n$ is $\G_{2n}$-equivariant, where $\G_{2n}$ is as defined above choosing the characterizing point for $\G_{2n}$ to be $p$;
\item the degree of the (positive) Gauss lift $a:T_n\to S^2$ is $(n-1)$;
\item the nullity of $T_n$ satisfies $\Null(T_n)\geq1$;
\item the (Morse) index of $T_n$ satisfies $\Ind(T_n)\geq 2n-3$.
\end{enumerate}
	
\end{theorem}
 \begin{remark}
 	The surface $\Sigma$ of \cref{thm: Main theorem} coincides with the suface $T_2$ of the above theorem.
 \end{remark}

Combining \cref{thm: multiple periods} and \cref{thm: saddle tower}, we can construct unstable minimal surfaces of any topological type in $(X^d_n,g_d)$.

\begin{corollary}
Let $n\geq2$ and $(X^d_n,g_d)$ be the multi-taub-NUT space with $m=1$ and characterizing points $p_1,...,p_{2n}$ lying on the vertices of a $2n$-sided regular polygon centered at the origin and such that $\abs{p_i}=d$ for every $i=1,...,2n$. For any $g\geq 0$, there exists $d_0>0$ such that, for every $f>d_0$ and $p\in\pi^{-1}(0)\cong S^1$, there is a minimal surface $T^g_n\subset X^d_n$ with the following properties:
\begin{enumerate}
	\item the surface $T^g_n$ is a Riemann surface of genus $g$;
\item the surface $T^g_n$ is $\G_{2n}$-equivariant;
\item the degree of the (positive) Gauss lift $a:T_n\to S^2$ is $(n+g-1)$;
\item the nullity of $T_n$ satisfies $\Null(T^g_n)\geq1$;
\item the (Morse) index of $T_n$ satisfies $\Ind(T^g_n)\geq 2(n+g)-3$.
\end{enumerate}
\end{corollary}

\subsection{Fueter maps and triholomorphic maps}
As described in the introduction, every minimal 2-sphere in a hyperk\"ahler 4-manifold $(M^4, g, J_1, J_2, J_3)$ with degree-1 positive Gauss lift can be parametrised by a harmonic map $u\co S^2\to M$ that satisfies the first-order PDE \cref{eq:Fueter}, i.e., 
\[
d_x u\circ J_{S^2}+\sum_{i=1}^3 x_i J_i\circ d_xu=0,
\]
where $J_{S^2}$ is the standard complex structure on $S^2$ and $x=(x_1, x_2,x_3)\in S^2\subset\R^3$. Thus each of the minimal surfaces constructed in \cref{thm: Main theorem,thm: spheres in K3} correspond to such special harmonic maps.

The radial extension $f$ to $\R^3$ of any solution $u$ to \cref{eq:Fueter} is a Fueter map $f\co \R^3\to M$ with an isolated singularity at the origin, i.e., it satisfies the \emph{Fueter} equation
\[
\sum_{i=1}^3{J_i\,  \frac{\partial f}{\partial x_i}}=0
\]
on $\R^3\setminus\{ 0\}$. Moreover, the extended map $F\co \R^4\to M$ defined by $F(x_0,x)=f(x)$ is \emph{triholomorphic} outside of its 1–dimensional singular set $\R\times\{0\}$, i.e.,
\[
dF = J_1 \circ dF \circ I + J_2 \circ dF \circ J + J_3 \circ dF \circ K,
\]
where $(I,J,K)$ is the standard triple of complex structures on $\R^4\simeq \mathbb{H}$. Since $f$ and $F$ are invariant under scalings in the domain they provide examples of tangent maps for Fueter and triholomorphic maps respectively. Besides realising systems of first-order PDEs that imply minimisation of the Dirichlet energy, these equations (and their generalisations to the case of sections of a non-trivial bundle with hyperk\"ahler fibres) arise in different contexts related to gauge theory in higher dimensions and the definition of enumerative invariants of 3-manifolds and hyperk\"ahler manifolds (see, e.g., \cite[\S 6]{Donaldson:Segal}). For these geometric applications it is essential to understand the compactness question for Fueter and triholomorphic map: the existence of Fueter tangent maps with an isolated codimension-3 singularity was stressed by Bellettini--Tian \cite{Bellettini:Tian}, though it was more recently realised by Esfahani \cite{Esfahani} that when the hyperk\"ahler target is 4-dimensional the topological constraint \cref{eq:Webster} implies that the tangent maps corresponding to minimal 2-sphere with degree-1 positive Gauss lift can not arise as limits of smooth Fueter maps. 

Nonetheless, we observe explicitly that, because of symmetries, the Fueter and triholomorphic maps arising from the minimal 2-spheres in \cref{thm: Main theorem,thm: spheres in K3} are stationary harmonic. Indeed, it is well-known (see, e.g., \cite[Lemma 3.8]{Chen:Li}) that this is the case if and only if the map $u\co S^2\to M$ is balanced, i.e.,
\[
\int_{S^2}{x_i|\nabla u|^2} =0
\]
where $(\omega_1,\omega_2,\omega_3)$ is the hyperk\"ahler triple on $M$. Using \cref{eq:Fueter} to rewrite
\[
\int_{S^2}{x_i|\nabla u|^2} =-2\int_{S^2}{u^\ast\omega_i},
\]
in view of \cref{rmK:Symmetries:hk:triple} it is then clear that the invariance under the discrete symmetry groups in \cref{thm: Main theorem,thm: spheres in K3} forces this balancing condition. 

\bibliography{refs}

\printaddress
\end{document}